\documentclass[12pt,psamsfonts,reqno]{amsart}
\usepackage{amsmath,amsthm,amssymb,amscd,amsfonts,amsbsy}
\usepackage{subfigure,graphicx}
\usepackage{afterpage,rotating,pdflscape}
\usepackage{color}
\usepackage{cite} 
\usepackage{overpic}
\usepackage{epstopdf}
\usepackage{hyperref,url}
\usepackage{array,multirow,multicol}
\usepackage{adjustbox,blindtext}
\usepackage[dvipsnames]{xcolor}
\usepackage{float}
\usepackage{mathabx}
\usepackage{stackengine}[2013-09-11]
\usepackage{scalerel}
\newcommand*{\defeq }{\mathrel{\vcenter{\baselineskip0.5ex \lineskiplimit0pt
                     \hbox{\scriptsize.}\hbox{\scriptsize.}}}%
                     =}
                     
\stackMath
\def\ptt{\mathrel{\ThisStyle{\raisebox{-.2ex}{$\SavedStyle\scalerel*%
    {\stackinset{c}{}{t}{.5ex}{\smash{-}}{\pitchfork}}{\pitchfork}$}}}}

\newcolumntype{R}{>{\displaystyle}r}
\newcommand{\R}{\ensuremath{\mathbb{R}}}

\newcommand{\CO}{\ensuremath{\mathcal{O}}}

\newcommand{\U}{\ensuremath{\mathcal{U}}}

\newcommand{\CF}{\ensuremath{\mathcal{F}}}
\newcommand{\ov}{\overline}

\newcommand{\T}{\theta}
\newcommand{\f}{\varphi}
\newcommand{\al}{\alpha}
\newcommand{\la}{\lambda}

\newcommand{\sgn}{\mathrm{sign}}

\makeatletter
\newcommand{\tpitchfork}{%
	\raise-0.3ex\vbox{
		\baselineskip\z@skip
		\lineskip-.52ex
		\lineskiplimit\maxdimen
		\m@th
		\ialign{##\crcr\hidewidth\smash{$-$}\hidewidth\crcr$\pitchfork$\crcr}
	}%
}

\def\p{\partial}
\def\e{\varepsilon}

\newtheorem {theorem} {Theorem} 

\newtheorem {proposition} {Proposition}

\newtheorem {lemma} {Lemma}
\newtheorem {remark}{Remark}
\newtheorem {definition} {Definition}

\newtheorem {mtheorem} {Theorem}

\usepackage{mathpazo}
\begin{document}

\title[Limit cycles in regularized Filippov systems]{On limit cycles in regularized Filippov systems\\ bifurcating from homoclinic-like connections\\ to regular-tangential singularities}
\author[D.D. Novaes and G. Rond\'{o}n]
{Douglas D. Novaes and Gabriel Rond\'{o}n}

\address{Departamento de Matem\'{a}tica, Instituto de Matem\'{a}tica, Estat\'{i}stica e Computa\c{c}\~{a}o Cient\'{i}fica (IMECC), Universidade
Estadual de Campinas (UNICAMP), Rua S\'{e}rgio Buarque de Holanda, 651, Cidade Universit\'{a}ria Zeferino Vaz, 13083--859, Campinas, SP, Brazil}
\email{ddnovaes@unicamp.br}

\address{UNESP - Universidade Estadual Paulista, S\~{a}o Jos\'{e} do Rio Preto, S\~{a}o Paulo, Brazil}
\email{gabriel.rondon@unesp.br}



\maketitle

\begin{abstract}
In this paper, we are concerned about smoothing of Filippov systems around homoclinic-like connections to regular-tangential singularities. We provide conditions to guarantee the existence of limit cycles bifurcating from such connections. Additional conditions are also provided to ensured the stability and uniqueness of such limit cycles. All the proofs are based on the construction of the first return map of the regularized Filippov system around homoclinic-like connections. Such a map is obtained by using a recent characterization of the local behaviour of the regularized Filippov system around regular-tangential singularities. Fixed point theorems and Poincar\'{e}-Bendixson arguments are also employed.
\end{abstract}

\section{Introduction}

The increasing interest in the study of piecewise smooth vector fields arises naturally due to the number of applications in different areas of knowledge, such as engineering, physics, economics, and biology. For more details see the books \cite{MR2368310,Mike} and the references therein. 

In this work, we are interested in planar piecewise smooth vector fields. More precisely, let $\mathcal{D}$ be an open set in $\R^2$ and $h:\mathcal{D}\to\R$ be a
$C^\infty$ function having 0 as a regular value. We assume for simplicity that
$\Sigma=h^{-1}(0)$ has a single connected component so that $M\setminus \mathcal{D}$ has two
connected components, denoted by $\Sigma^+=h^{-1}(0,\infty)$ and $\Sigma^-=h^{-1}(-\infty,0)$. Accordingly, the piecewise smooth vector field in $\R^2$ is defined by
\begin{equation}\label{locdds}
Z(p)=(X^+,X^-)_{\Sigma}=\left\{\begin{array}{l}
X^+(p),\quad\textrm{if}\quad p\in\Sigma^+,\vspace{0.1cm}\\
X^-(p),\quad\textrm{if}\quad p\in\Sigma^-,
\end{array}\right. \quad \text{for}\quad p\in \mathcal{D}.
\end{equation}
We suppose that the vector fields $X^+$ and $X^-$ have an extension to $\Sigma$ which is, at least $C^2$. 

Throughout this paper we will denote the components of $X^\pm$ by $X^\pm_i$ for each $i\in\{1,2\},$ i.e. $X^\pm=(X^\pm_1,X^\pm_2)$.

\subsection{Filippov's Convention}
In \cite{Filippov88}, the local trajectories of piecewise smooth vector fields \eqref{locdds} were defined as solutions of a differential inclusion $\dot p\in\CF_Z(p).$ This approach is called Filippov's convention. The piecewise smooth vector field \eqref{locdds} is called Filippov system when its dynamics is provided by the Filippov's convention. 

In order to illustrate the Filippov's convention we define the following open regions on $\Sigma:$
\begin{enumerate}
\item[(i)] $\Sigma^s=\{p\in\Sigma:\, X^+h(p)\cdot X^-h(p) < 0\}$,

\item[(ii)] $\Sigma^c=\{p\in \Sigma:\, X^+h(p)\cdot X^-h(p) > 0\}$,
\end{enumerate}
where $X^{\pm}h(p)=\langle\nabla h(p),X^{\pm}(p)\rangle$ denotes the Lie derivative  of $h$ in the direction of the vector fields $X^{\pm}.$ In addition, $(X^{\pm})^ih(p)=X^{\pm}((X^{\pm})^{i-1}h)(p)$ for $i>1.$ 

The set $\Sigma^c$ is called {\it crossing region}. For $p\in\Sigma^c,$ the trajectories either side of the discontinuity $\Sigma,$ reaching $p,$ can be joined continuously, forming a trajectory that crosses $\Sigma^c.$

The set  $\Sigma^{s}$ is called {\it sliding region} and both vectors fields $X^+$ and $X^-$ simultaneously point inward to or outward from $\Sigma.$ We consider the following {\it sliding vector field} defined on $\Sigma^{s}:$ 
\begin{equation*}
Z^s(p)= \dfrac{X^- h(p) X^+(p)- X^+ h(p) X^-(p)}{X^- h(p) - X^+ h(p) },\,\, \text{for} \,\, p\in \Sigma^{s}.
\end{equation*}
Thus, the trajectories either side of the discontinuity $\Sigma$ reaching $p\in \Sigma^{s}$  can be joined continuously to trajectories that slide on $\Sigma^{s}$ following $Z^s(p)$. In addition, a singularity of the sliding vector field $Z^s$ is called pseudo-equilibrium. 

In the Filippov context, the notion of {\it$\Sigma$-singular points} also comprehends the tangential points $\Sigma^t$ constituted by the contact points between $X^+$ and $X^-$ with $\Sigma,$ i.e. $\Sigma^t=\{p\in \Sigma:\, X^+h(p)\cdot X^-h(p) = 0\}.$ In this paper, we are interested in contact points of finite degeneracy. Recall that $p$ is a {\it contact of order $k-1$} (or multiplicity $k$) between $X^\pm$ and $\Sigma$ if $0$ is a root of multiplicity $k$ of $f(t)\defeq h\circ \f_{X^\pm}(t,p),$ where $t\mapsto \f_{X^\pm}(t,p)$ is the trajectory of $X^\pm$ starting at $p.$ Equivalently, $$X^\pm h(p) = (X^\pm)^2h(p) = \ldots = (X^\pm)^{k-1}h(p) =0,\text{ and } (X^\pm)^{k} h(p)\neq 0.$$
In addition, an even multiplicity contact, say $2k,$ is called {\it visible} for $X^+$ (resp. $X^-$) when $(X^{+})^{2k}h(p)>0$ (resp. $(X^{-})^{2k}h(p)<0$). Otherwise, it is called {\it invisible}. 

In this work, we shall focuses our attention in {\it visible regular-tangential singularities of multiplicity $2k$}, which are formed by a visible even multiplicity of $X^+$ and a regular point of $X^-,$ or vice versa (see Figure \ref{fig2tan}).

\begin{figure}[H]
	\begin{center}
		\begin{overpic}[scale=0.63]{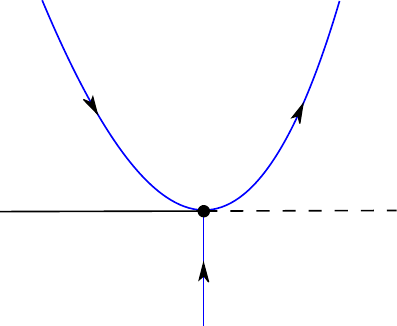}
		\end{overpic}
		\caption{\tiny{Visible regular-tangential singularity.}}
	\label{fig2tan}
	\end{center}
	\end{figure}

\subsection{Sotomayor-Teixeira Regularization}
Frequently, non-smooth mathematical models are discontinuous idealizations of regular phenomena. Thus, in order to investigate a regular phenomenon with non-smooth mathematical model, it seems natural to inquire how the solutions of smooth systems, which converge to the non-smooth one, behave in the limit. Therefore, it is usual to consider a non-smooth system as a singular limit of a 1-parameter family of smooth vector fields. 

The concept of smoothing of a piecewise smooth vector field $Z$  was firstly introduced by Sotomayor and Teixeira in \cite{ST1995} and consists in obtaining a one-parameter family of smooth vector fields $Z_{\e}$ converging to $Z$ when $\e$ converges to $0$. More preciselly, consider a $C^{\infty}$ function $\phi:\R\rightarrow\R$ satisfying $\phi(\pm1)=\pm1,$ $\phi^{(i)}(\pm1)=0$ for $i=1,2,\ldots,n,$ and $\phi'(s)>0$ for $s\in(-1,1).$ Then, a {\it $C^n$-monotonic transition function} is defined as follows:
\begin{equation}\label{Phi}
\Phi(s)=\left\{\begin{array}{ll}
\phi(s)&\text{if}\quad|s|\leqslant 1,\\
\sgn(s)&\text{if}\quad|s|\geqslant1.
\end{array}\right.
\end{equation}
The set of $C^{n}$-monotonic transition functions $\Phi$ which are not $C^{n+1}$ at $\pm 1$ is denoted by $C^{n}_{ST}.$ Hence, the Sotomayor-Teixeira $C^n$-regularization of $Z$ is the following one--parameter family of smooth vector fields $Z^{\Phi}_{\e}$
\begin{equation}\label{regula}
Z_{\e}^{\Phi}(p)=\dfrac{1+\Phi(h/\e)}{2}X^+(p)+\dfrac{1-\Phi(h/\e)}{2}X^-(p), \, \text{for}\, \e>0.
\end{equation}
Notice that the regularized vector field $Z_{\e}^{\Phi}(p)$ coincides with $X^+(p)$ or $X^-(p)$ whether $h(p)\geqslant\e$ or $h(p)\leqslant-\e,$ respectively. In the region $|h(p)|\leqslant \e,$ the vector $Z_{\e}^{\Phi}(p)$ is a linear combination of $X^+(p)$ and $X^-(p).$ 

\subsection{$\Sigma$-Polycycles in Filippov Systems}
In these last decades, the study of $\Sigma-$polycycles  in Filippov systems have been considered in several papers. For instant, in \cite{MR2012652} the authors 
introduced the critical crossing cycle bifurcation, which is defined as a one-parameter family $Z_\alpha$ of Filippov systems, where $Z_0$ has a homoclinic-like connection to a fold-regular singularity. In \cite{MR3401591}, Freire et al. proved that the unfolding of a critical crossing cycle bifurcation provided in \cite{MR2012652} holds in a generic scenario. More degenerate homoclinic-like connections to $\Sigma$-singularities has also been considered. In \cite{MR3816665}, the authors studied a codimension-two homoclinic-like connection to a visible-visible fold-fold singularity. In \cite{NovSeaTeiZel2020} its unfolding under non-autonomous periodic perturbation has been adressed.  Recently, in \cite{AndGomNov19}, Andrade et. al. developed a rather general method to investigate the unfolding of $\Sigma$-polycycles in Filippov systems. This method was applied to describe bifurcation diagrams of Filippov systems around several $\Sigma$-polycycles. The readers are referred to \cite{MR3040394,MR3022071,MR2763562,MR1960724,MR3244478} for more studies on $\Sigma-$polycycles.
The interest in studying $\Sigma-$polycycles is due to the fact that they are non-local invariant sets that provide information on the dynamics of the system.

In what follows, we shall introduce some basic concepts for this article. 
First, we  define the local separatrix at a point $p\in\Sigma $.

\begin{definition}
If $p\in\Sigma$, the \textbf{asymptotically stable (resp. unstable) separatrix} $W^s_{t,\ptt}(p),$ (resp. $W^u_{t,\ptt}(p))$ of $Z=(X^+,X^-)$ at a visible regular-tangential singularity $p$ in $\Sigma^\pm$ is defined as
\[\begin{array}{rcl}
W^{s,u}_t(p)&=&\{q=\varphi_{X^+}(t(q),p): \varphi_{X^+}(I(q),p)\subset\Sigma^+ \hspace{0.2cm}\text{and}\hspace{0.2cm} \delta_{s,u}t(q)>0\},\\
W^{s,u}_{\ptt}(p)&=&\{q=\varphi_{X^-}(t(q),p): \varphi_{X^-}(I(q),p)\subset\Sigma^- \hspace{0.2cm}\text{and}\hspace{0.2cm} \delta_{s,u}t(q)>0\},
\end{array}\] where, $\delta_u=1$, $\delta_s=-1$, and $I(q)$ is the open interval with extrema 0 and $t(q)$.
\end{definition}

Now, we introduce the concepts of regular orbit and $\Sigma$-polycycle for planar Filippov systems.

\begin{definition}\label{defpoly} Consider the Filippov system $Z=(X^+,X^-).$
\begin{enumerate}
  \item[(i)] We say that $\gamma$ is a \textbf{regular orbit} of $Z$ if it is a piecewise smooth curve such that $\gamma\cap \Sigma^+$ and $\gamma\cap \Sigma^-$ are unions of regular orbits of $X^+$ and  $X^-$, respectively, and $\gamma\cap\Sigma\subset\Sigma^c$.

	\item[(ii)] A closed curve $\Gamma$ is said to be a \textbf{$\Sigma-$polycycle} of $Z$ if it is composed by a finite number of $\Sigma-$singularities, $p_1, p_2,\cdots,p_n\in\Sigma,$ and a finite number of regular orbits of $Z$, $\gamma_1, \gamma_2,\cdots,\gamma_n$, such that for each $1\leq i\leq n$, $\gamma_i$ has ending
points $p_i$ and $p_{i+1}$, and satisfies that
    \begin{enumerate}
	     \item $\Gamma$ is a $S^1$-immersion and it is oriented by increasing time along the regular orbits;
       \item there exists a non-constant first return map $\pi_\Gamma$ defined, at least, in one side of $\Gamma$.
    \end{enumerate}

\end{enumerate}
\end{definition}

\begin{remark}
Condition $(a)$ in Definition \ref{defpoly} gives the minimality of $\Sigma-$polycycles of $Z,$ i.e. a $\Sigma-$polycycle $\Gamma$ can not be written as union of two or more $\Sigma-$polycycles. This condition also establishes the notion of sides of $\Gamma$, which is used in Condition $(b).$
\end{remark}

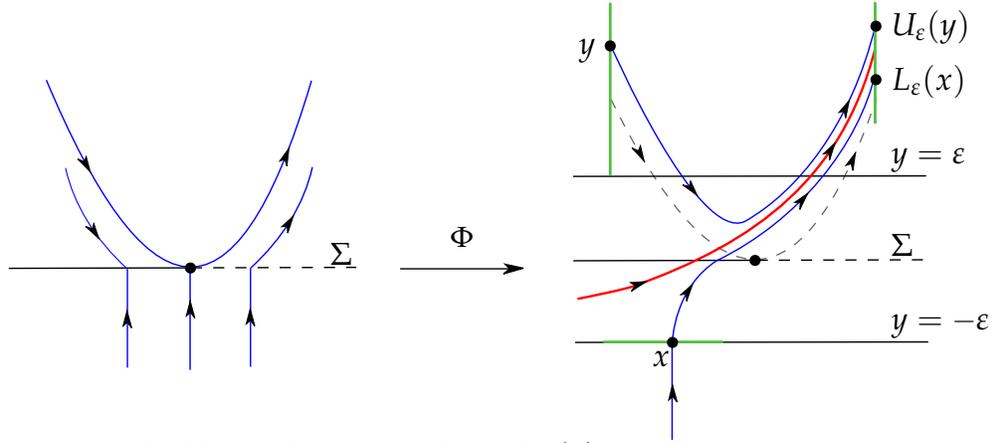
\begin{figure}[H]
	\begin{center}
		\begin{overpic}[scale=0.7]{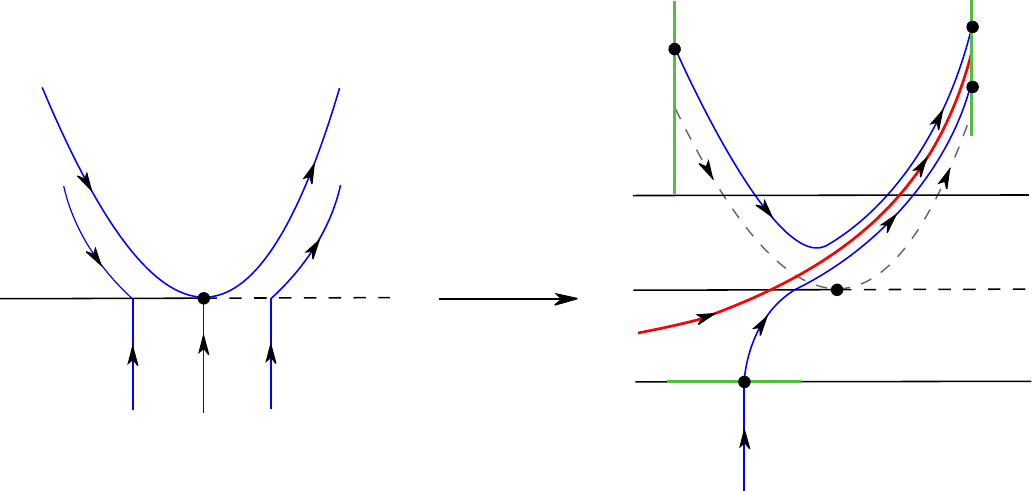}
		\put(70,8){$x$}
		\put(62,42){$y$}
		\put(96,38){$L_\e(x)$}
		\put(96,44){$U_\e(y)$}
\put(96,20){$\Sigma$}
		 \put(96,30){$y=\e$}
		\put(96,12){$y=-\e$}
		\put(48,21){{\small $\Phi$}}
		\put(35,19){$\Sigma$}
		\end{overpic}
		\caption{\tiny{Upper Transition Map $U_{\e}(y)$ and Lower Transition Map $L_{\e}(y)$ defined for $C^n$-regularizations of Filippov systems around visible regular-tangential singularities of multiplicity $2k$.}}
	\label{figreg}
	\end{center}
	\end{figure}

We remark that the simplest $\Sigma-$polycycle is the boundary limit cycle. In \cite{BonetSeara16}, Bonet and Seara studied $C^n$-regularizations of boundary limit cycles with a local fold singularity with the switching manifold $\Sigma$. In \cite{NovRon2019}, Novaes and Rond\'on generalized the results of [1] for contacts of even multiplicity 
with the switching manifold $\Sigma$. For this, the authors investigated how the trajectories of the regularized system $Z_{\e}^{\Phi}$ transits around a visible regular-tangential singularity of multiplicity $2k$. They proved that the flow of the regularized system defines transition maps between some traversal sections near the regular-tangential singularity, namely the Upper Transition Map $U_{\e}(y)$ and the Lower Transition Map  $L_{\e}(y)$ (see Figure \ref{figreg}).

Thus, our main interest in this study consists in generalizing the results obtained in \cite{BonetSeara16,NovRon2019} for a type of $\Sigma-$polycycles, which we call homoclinic-like $\Sigma-$polycycles through a unique $\Sigma-$singularity of regular-tangential type of planar Filippov systems (see Figure \ref{figtan3}). In what follows, we divide this special type of $\Sigma-$polycycles into 2 classes.
In order to define them, consider the Filippov system $Z=(X^+,X^-)$ and suppose that $Z$ has a $\Sigma-$polycycle $\Gamma$ having a unique visible regular-tangential singularity at $p$ of multiplicity $2k,$ for $k\geq 1$. Through a local change of coordinates, we can assume that $p=(0,0)$ and $h(x,y)=y$. Now, we shall locally characterize $\Gamma$ around $p$. For this, it is enough to divide the homoclinic-like $\Sigma-$polycycles at $p$ of $Z$ in the following 2 classes:
\begin{enumerate}
	 \item[(a)] 
	$W^s_t(p)\cup W^u_t(p)\subset \Gamma$ and $X^-h(p)>0$ or $X^-h(p)<0$;
	 \item[(b)] 
	 $W^s_{\ptt}(p)\cup W^u_t(p)\subset \Gamma$ or $W^u_{\ptt}(p)\cup W^s_t(p)\subset \Gamma$.  
	\end{enumerate}
Here, $\Sigma$-polycycles satisfying condition $(a)$ or $(b)$ will be called $\Sigma$-polycycles of type $(a)$ or type $(b)$, respectively (see Figure \ref{figtan3}).
 \begin{figure}[h]
	\begin{center}
    \begin{overpic}[scale=0.3]{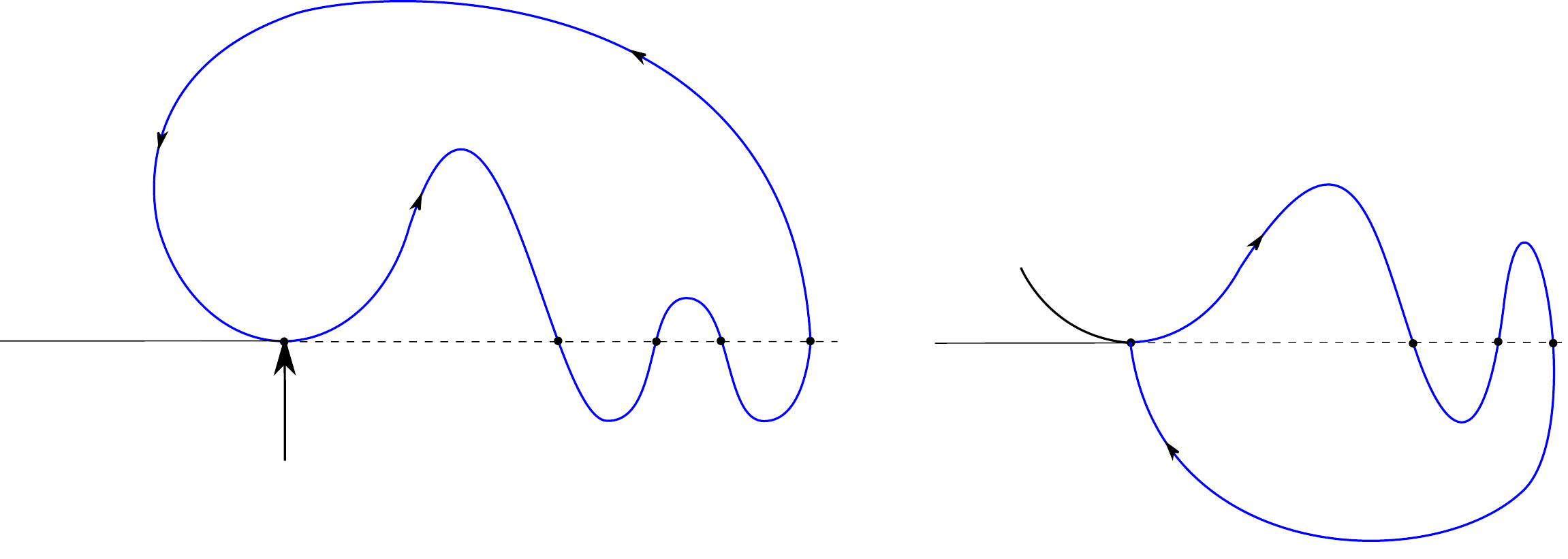}
		\put(72,15){$p$}
		\put(88,11){$q_1$}
		\put(95.5,11){$q_2$}
		\put(99.5,11){$q_3$}
		\put(17,15){$p$}
		\put(33,11){$q_1$}
		\put(38.5,15){$q_2$}
		\put(43,11){$q_3$}
		\put(30,-3){$(a)$}
		\put(83,-3){$(b)$}
		\put(48,15){$q_4$}
    \put(89,20){$\Gamma$}
		 \put(100.5,14){$\Sigma$}
		\put(54.5,14){$\Sigma$}
		\put(45,30){$\Gamma$}
		\end{overpic}
		\smallskip
		\caption{\tiny{Examples of homoclinic-like $\Sigma-$polycycles $\Gamma$ of type $(a)$ and $(b)$, respectively.}}
	\label{figtan3}
	\end{center}
	\end{figure}
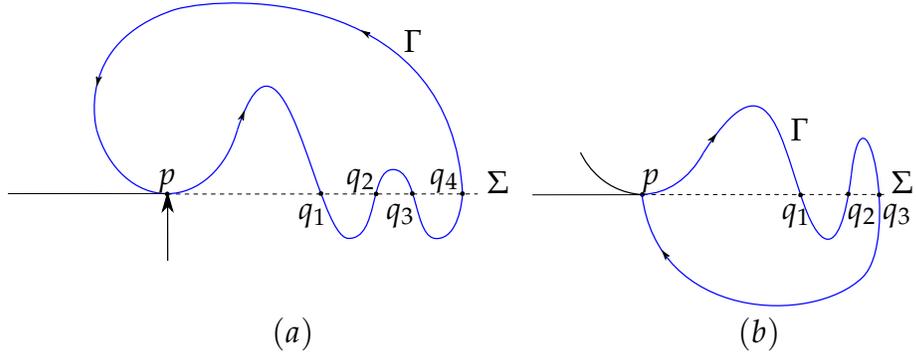
	
 In addition, we investigate polycycles having several contacts of type $(a)$ and type $(b)$ which we call polycycles of type $(i)$ and type $(ii)$, respectively.
\subsection{Main Results and Structure of the Paper}

The first main result of this paper (Theorems A) is concerned with regularization of Filippov systems having a $\Sigma-$polycycle of type $(a)$. It establishes sufficient conditions for the existence of limit cycles of the regularized system $Z_\e^\Phi$ passing through of a certain compact set with nonempty interior. When the limit cycle exists, its stability is characterized and its convergence to the $\Sigma-$polycycle is ensured. Theorem A is stated and proven in Section \ref{sec:Polycycles}. In Section \ref{sec:nonexistence}, we study cases of uniqueness and nonexistence of limit cycles for the regularization of $\Sigma-$polycycles of type $(a)$. More specifically, in Proposition \ref{proptc1} we provide a class of piecewise smooth vector fields having  a $\Sigma-$polycycle of type $(a)$ for which its regularization either does not admit limit cycle or admit a unique limit cycle converging to the $\Sigma-$polycycle.

The second main result of this paper (Theorem B) is concerned with regularization of Filippov systems having a $\Sigma-$polycycle of type $(b)$. It establishes sufficient conditions for the existence of limit cycles of the regularized system $Z_\e^\Phi$ converging to the $\Sigma-$polycycle. Theorem B is stated and proven in Section \ref{sec:polycyclesb}.

In Section \ref{sec:Polycyclesk}, we generalize our main results for the case of polycycles with a finite number of tangential-regular contacts (Theorems \ref{tc2_i} and \ref{tc2_ii}).

The proofs of Theorems \ref{tc1}, \ref{tc2}, \ref{tc2_i}, and \ref{tc2_ii} are mainly based on the characterization of the upper and lower transition maps around regular-tangential singularities, fixed point theorems, and \textit{Poincar\'{e}-Bendixson Theorem}.

\section{Regularization of $\Sigma$-Polycycles of type $(a)$.}\label{sec:Polycycles}
In this section, we shall state and prove the first main result of this paper, which in particular establishes sufficient conditions under which the regularized vector field $Z_{\e}^{\Phi}$ has a limit cycle $\Gamma_{\e}$ converging to a $\Sigma-$polycycle of type $(a).$ For this, suppose that a Filippov system $Z=(X^+,X^-)$ has a $\Sigma-$polycycle $\Gamma$ of type $(a)$. Through a local change of coordinates, we can assume that $p=(0,0)$ and $h(x,y)=y$. Without loss of generality, assume that:
	
	\begin{enumerate}
	\item[(a.1)]  $X_1^+(p)>0$;
	 \item[(a.2)]  the trajectory of $Z$ through $p$ crosses $\Sigma$ transversally $m-$times
at $q_1,\cdots, q_m$, i.e. if $m\neq 0,$ then for each $i = 0,\cdots, m$, there exists $t_i>0$ such
that $\varphi_Z(t_i,q_i)=q_{i+1}$, where $q_{m+1}=p=q_0$. Moreover, $\Gamma\cap\Sigma =\{q_1,\cdots, q_m, p\}$.
	\end{enumerate}
We shall also assume that 
\begin{enumerate}
\item[(a.3)] $ X^-h(p)>0.$
\end{enumerate}
The case $X^-h(p)<0$ is obtained from this case multiplying the vector field $Z$ by -1 (see Remark \ref{remark1} below).

Notice that assumption $(a.2)$ above guarantees the existence of a diffeomorphism $D$ associated to $Z,$ which we call exterior map. In what follows, we characterize such a map $D.$ Since $X_1^+(p)> 0$, implies that there exists an open set $U$ such that $X_1^+(x,y)>0,$ for all $(x,y)\in U.$ Take $\rho,\T>0$ small enough in order that the points $q^{u}=(\T,\ov{y}_\T)\in W^{u}_t(p)$ and $q^{s}=(-\rho,\ov{y}_{-\rho})\in W^{s}_t(p)$ are contained in $U$. Hence, there exist $\delta^{u,s}$ and $\eta$ positive numbers such that
\begin{equation}\label{tsec}
\begin{array}{l}
\tau^{u}_t=\{(\T,y):y\in(\ov{y}_\T-\delta^{u},\ov{y}_\T+\delta^{u})\},\\
\tau^{s}_t=\{(-\rho,y):y\in(\ov{y}_{-\rho}-\delta^{s},\ov{y}_{-\rho}+\delta^{s})\},\quad\text{and}\\
\sigma_p=\{0\}\times[0,\eta)
\end{array}\end{equation} are transversal sections of $X^+$.
In addition, since $X_2^\pm(q_i)\neq 0$ for each $i = 1,\cdots, m,$ there exist $\e_i>0$ such that $\tau_i=(q_i-\e_i,q_i+\e_i)\times\{0\}$ is a transversal sections of $X^\pm$. Hence, by the \textit{Tubular Flow Theorem} (see\cite{MR669541}) there exist the $C^{2k}-$diffeomorphisms $T^{u,s}:\sigma_p\longrightarrow\tau^{u,s}_t$ and $D_i:\tau_{i-1}\longrightarrow\tau_{i}$ for each $i = 1,\cdots, m+1,$ such that $T^{u,s}(p)=q^{u,s}$ and $D_1(q^u)=q_1,$ $D_{m+1}(q_m)=q^s,$ and $D_i(q_{i-1})=q_{i}$ for each $i = 2,\cdots, m,$ where $\tau_0=\tau^u_t$ and $\tau_{m+1}=\tau^s_t$ (see Figure \ref{figtan1}). Thus, defining the $C^{2k}-$diffeomorphism $D:=D_{m+1}\circ\cdots\circ D_1$ and expanding $D$ around $y=\ov{y}_\T$, we get
 \begin{equation}\label{Di}
  D(y)=\ov{y}_{-\rho}+r_{\T,\rho}(y-\ov{y}_\T)+\CO((y-\ov{y}_\T)^2),
 \end{equation}
where $r_{\T,\rho}=\frac{d D}{dy}(\ov{y}_\T).$
 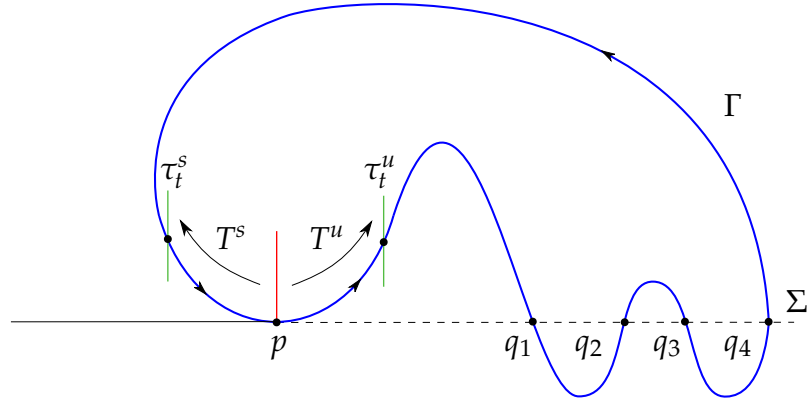
\begin{figure}[h]
	\begin{center}
    \begin{overpic}[scale=0.5]{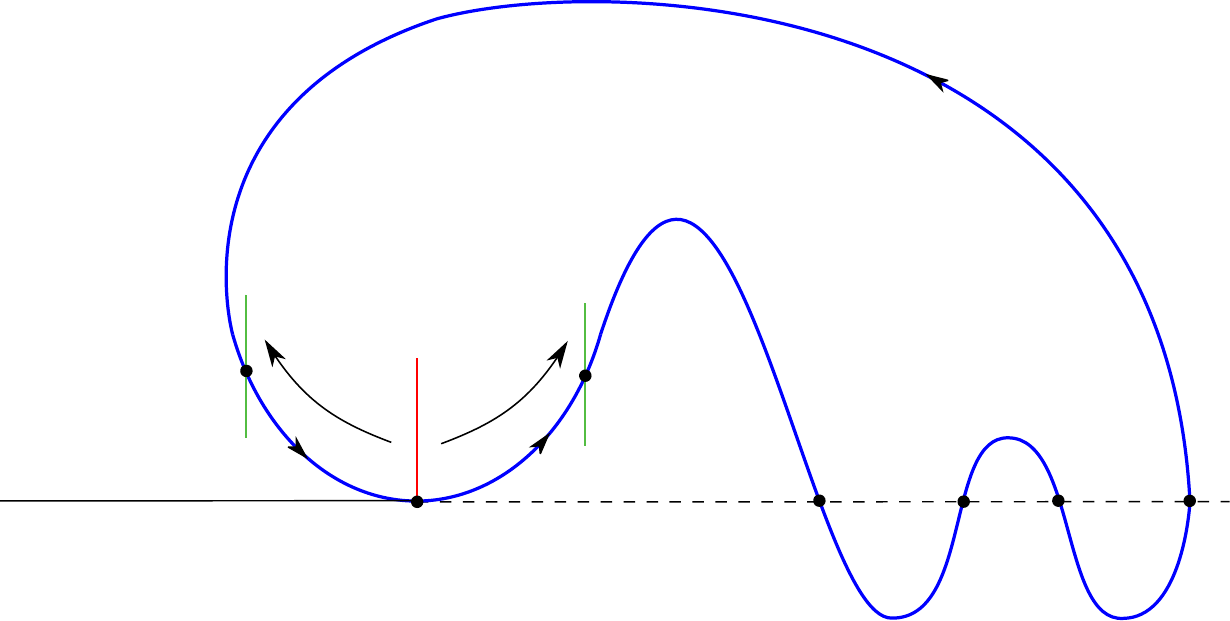}
		\put(33,6){$p$}
		\put(82,6){$q_3$}
		\put(91,6){$q_4$}
		\put(63,6){$q_1$}
		\put(72,6){$q_2$}
		\put(26,19){$T^s$}
		\put(38,19){$T^u$}
    \put(19,28){$\tau^s_{t}$}
		 \put(45,28){$\tau^u_{t}$}
		\put(99,11){$\Sigma$}
		\put(91,36){$\Gamma$}
	    \put(49,19){\scriptsize $q^u$ \par}
	    \put(16,19){\scriptsize $q^s$ \par}
		\end{overpic}
		\caption{\tiny{$\Sigma-$polycycle $\Gamma$ of type $(a)$ satisfying $(a.1),$ $(a.2),$ and $(a.3)$.}}
	\label{figtan1}
	\end{center}
	\end{figure}
	
Now, from Definition \ref{defpoly}, we know that there exists a first return map $\pi_\Gamma$ defined, at least, in one side of $\Gamma.$ In what follows, we shall see that there exist constants $\eta>0$ and $K>0$ such that the first return map $\pi_\Gamma:\sigma_p\rightarrow\sigma_p,$ with $\sigma_p=\{0\}\times[0,\eta),$ is given by 
	\begin{equation}\label{piK}
	\pi_\Gamma(y)=Ky+\CO(y^2).
	\end{equation}
Indeed, expanding $T^{u,s}$ around $y=0,$ we have 
\begin{equation}\label{Tsud}
\begin{split}
T^{s}(y)&=\ov{y}_{-\rho}+\kappa_{\rho}^{s}y+\CO(y^{2}) \quad\text{and}\\
T^{u}(y)&=\ov{y}_\T+\kappa_{\T}^{u}y+\CO(y^{2}),\\
\end{split}
\end{equation}
with $\kappa_{\rho,\T}^{s,u}=\frac{d T^{s,u}}{dy}(0).$ In addition, using \eqref{Di} and the \textit{Implicit Function Theorem}, we get
\begin{equation}\label{Dinv1}
D^{-1}(y)=\ov{y}_\T+\frac{1}{r_{\T,\rho}}(y-\ov{y}_{-\rho})+\CO((y-\ov{y}_{-\rho})^2).
\end{equation}
Since $X^+$ and $X^-$ are planar vector fields, the uniqueness of solutions implies that  $r_{\T,\rho}, \kappa_{\rho,\T}^{s,u}>0$  for all $\T,\rho$ sufficiently small.
Hence, we can define the first return map $\pi_\Gamma:\sigma_p\longrightarrow\sigma_p$ by 
	$$\pi_\Gamma(y)=(D^{-1}\circ T^s)^{-1}\circ T^u(y).$$ 
	From \eqref{Tsud} and \eqref{Dinv1}, we obtain $D^{-1}\circ T^s(y)=\ov{y}_\T+\frac{\kappa_{\rho}^s}{r_{\T,\rho}}y+\CO(y^{2}).$ Thus, using the \textit{Implicit Function Theorem}, we have that
	\begin{equation}\label{DTinv1}
	(D^{-1}\circ T^s)^{-1}(y)=\frac{r_{\T,\rho}}{\kappa_\rho^s}(y-\ov{y}_\T)+\CO((y-\ov{y}_\T)^2).
	\end{equation}
	Therefore, using \eqref{Tsud} and \eqref{DTinv1} we conclude that 
	\begin{equation}\label{pinK}
	\pi_\Gamma(y)=\frac{r_{\T,\rho}\kappa_\T^u}{\kappa_\rho^s}y+\CO(y^2).
	\end{equation}
	Consequently, taking $K:=\frac{r_{\T,\rho}\kappa_\T^u}{\kappa_\rho^s}>0,$ we get \eqref{piK}.

Now, since $D$ is a diffeomorphism induced by a regular orbit, we can easily see that the regularized system  $Z_{\e}^{\Phi}$ also admits an exterior map  $D_\e:\tau^{u}_t\longrightarrow\tau^{s}_t$. If we denote $S:=\frac{\partial D_\e(0)}{\partial\e}\Big|_{\e=0},$ then
\begin{equation}\label{Teps}
D_\e(y)=D(y)+S\e+\CO(\e^2,\e y),
\end{equation}
See Appendix A for more details about the exterior map and how to estimate $S$.

In what follows, we state our first main result, which will be proven in Section \ref{sec:prooftc1}.
\begin{mtheorem}\label{tc1}
Consider a Filippov system $Z=(X^+,X^-)_{\Sigma}$ and assume that $Z$ has a $\Sigma-$polycycle $\Gamma$ of type $(a)$ satisfying $(a.1),$ $(a.2),$ and $(a.3)$. For $n\geq 2k-1,$ let $\Phi\in C^{n-1}_{ST},K,S$ be given as $\eqref{Phi},$ $\eqref{piK}$ and $\eqref{Teps}$ respectively and consider the regularized system $Z_{\e}^{\Phi}$ \eqref{regula}. If $K+S-1\neq 0,$ then the following statements hold:
\begin{enumerate}
	\item[(a)] Given $0<\la<\la^*=\frac{n}{1+2k(n-1)},$ if $K+S-1>0$, then there exists $\rho>0$ such that  the regularized system $Z_{\e}^{\Phi}$ does not admit limit cycles passing through the section  $\widehat H_{\rho,\la}^{\e}=[-\rho,-\e^{\la}]\times\{\e\},$ for $\e>0$ sufficiently small.
	\item[(b)] Given $\frac{1}{2k}<\la<\la^*=\frac{n}{1+2k(n-1)},$ if $K+S-1<0$, then there exists $\rho>0$ such that the regularized system $Z_{\e}^{\Phi}$ admits a unique limit cycle $\Gamma_{\e}$ passing through the section $\widehat H_{\rho,\la}^{\e}=[-\rho,-\e^{\la}]\times\{\e\},$ for $\e>0$ sufficiently small. Moreover, $\Gamma_{\e}$ is asymptotically stable and $\e$-close to $\Gamma.$ 
\end{enumerate}
\end{mtheorem}

\begin{remark}\label{remark1}
The results given by Theorem \ref{tc1} depend on the topological type of the regular-tangential singularity (see \cite[Remark 5]{NovRon2019}). That is, the hypotheses of Theorem \ref{tc1} implies that any neighborhood of the regular-tangential singularity has non-empty intersection with the attracting sliding region $\Sigma^s$ of $\Sigma,$ i.e., with the set of points $q\in\Sigma$ such that $X^+h(q)<0$ and $X^-h(q)>0$ (see Figure \ref{fig_s_e}(a)).

If, in hypotheses of Theorem \ref{tc1}, we suppose that $X^-$ points outwards $\Sigma$ at $p,$ i.e., $X^-h(p)<0,$  then any neighborhood of the regular-tangential singularity has non-empty intersection with the repelling sliding region $\Sigma^e$ of $\Sigma,$ i.e., with the set of points  $q\in\Sigma$ such that $X^+h(q)>0$ and $X^-h(q)<0$ (see Figure \ref{fig_s_e}(b)). In this context, a version of Theorem \ref{tc1} can be obtained by multiplying the vector field by $-1$. Consequently, the limit cycle obtained by $Z,$ under this assumption, would be unstable.
\begin{figure}[h]
	\begin{center}
		\begin{overpic}[scale=0.6]{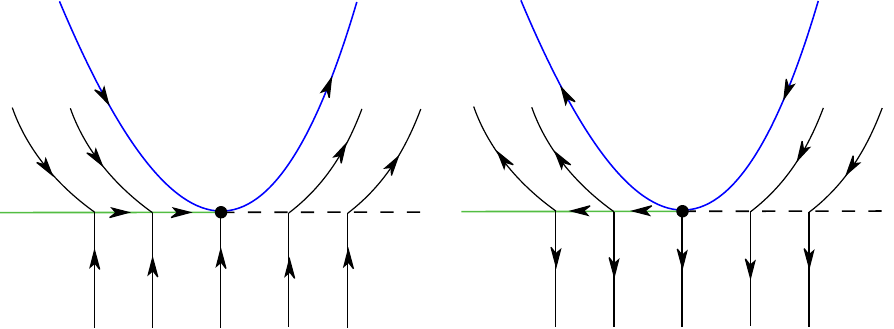}
	
\put(45,14){$\Sigma^c$}
\put(96,14){$\Sigma^c$}
\put(0,14){$\Sigma^s$}
\put(52,14){$\Sigma^s$}
	\put(24,-5){$(a)$}
	\put(76,-5){$(b)$}
		\end{overpic}
		\bigskip
		\caption{\tiny{Topological types of visible regular-tangential singularities of even multiplicity. In $(a)$, $X^-$ points towards $\Sigma$ at $p.$ In $(b)$, $X^-$ points outwards $\Sigma$ at $p.$}}
	\label{fig_s_e}
	\end{center}
	\end{figure}
\end{remark}

In order to prove this theorem, we shall first establish the relationship between the derivative of the first return map $K$ and the derivative of exterior map $r_{\T,\rho}$ as follows.

\begin{lemma}\label{rTR}
Consider $r_{\T,\rho}$ given as in \eqref{Di}. Then, $\lim\limits_{\T,\rho \to 0}r_{\T,\rho}=K.$
\end{lemma}
\begin{proof}
Using equations \eqref{piK} and \eqref{pinK} we get for $\T,\rho>0$ small enough that
$$r_{\T,\rho}=\frac{K\kappa_\rho^s}{\kappa_\T^u}.$$
To prove this lemma, we just need to prove that $\lim\limits_{\T,\rho \to 0}\frac{\kappa_\rho^s}{\kappa_\T^u}=1,$ because in this case we have 
\[\begin{array}{rl}
\lim\limits_{\T,\rho \to 0}r_{\T,\rho}=&\lim\limits_{\T,\rho \to 0}\frac{K\kappa_\rho^s}{\kappa_\T^u}\\
=&K\lim\limits_{\T,\rho \to 0}\frac{\kappa_\rho^s}{\kappa_\T^u}\\
=&K.
\end{array}\]
For this, we shall prove that the flow of $X^+$ induces $\mathcal{C}^{2
k}$ maps,
$$\la_\rho^s:\sigma'_p\subset\sigma_p\longrightarrow\tau^s_t\quad\text{and}\quad \la_\T^u:\sigma'_p\subset\sigma_p\longrightarrow \tau^u_t,$$
between the transversal sections defined in \eqref{tsec} and satisfying $\la_\rho^s(0)=q^s$ and $\la_\T^u(0)=q^u$, respectively. Indeed, consider the functions 
$$\mu_{1}(t,y,\T)=\varphi^1_{X^+}(t,0,y)-\T, \hspace{0.2cm}\text{for} \hspace{0.2cm} (0,y)\in \sigma_p, \hspace{0.2cm} t\in I_{(0,y)},$$ and 
$$\mu_{2}(t,y,\rho)=\varphi^1_{X^+}(t,0,y)+\rho,\hspace{0.2cm}\text{for} \hspace{0.2cm} (0,y)\in \sigma_p,\hspace{0.2cm} t\in I_{(0,y)},$$ where $\varphi_{X^+}=(\varphi^1_{X^+},\varphi^2_{X^+})$ is the flow of $X^+$ and $I_{(0,y)}$ is the maximal interval of existence of $t\mapsto\varphi_{X^+}(t,0,y)$. Since, 
$$\mu_{1}(0,0,0)=0=\mu_{2}(0,0,0) \, \text{and}$$  
$$\frac{\partial \mu_{1,2}}{\partial t}(0,0,0)=\frac{\partial \varphi^1_{X^+}}{\partial t}(0,0,0)=X_1^+(p)\neq 0,$$ 
by the \textit{Implicit Function Theorem} there exist $\eta_0>0$ and smooth functions $t_{1}(y,\T)$ and $t_{2}(y,\rho),$ with $(0,y)\in\sigma'_p:=\{0\}\times[0,\eta_0)\subset\sigma_p$ and $\T,\rho>0$ sufficiently small, such that 
$$t_1(0,0)=0=t_2(0,0),$$ 
$$\mu_{1}(t_{1}(y,\T),y,\T)=0\quad \text{and} \quad\mu_{2}(t_{2}(y,\rho),y,\rho)=0,$$ 
i.e. $\varphi^1_{X^+}(t_1(y,\T),0,y)=\T$ and $\varphi^1_{X^+}(t_2(y,\rho),0,y)=-\rho.$ Thus, we can define the functions 
$$
\la_\T^u(y)=\varphi^2_{X^+}(t_{1}(y,\T),0,y)\quad \text{and}\quad\la_\rho^s(y)=\varphi^2_{X^+}(t_{2}(y,\rho),0,y).
$$
Notice that 
$$\frac{d\la_\T^u}{dy}(0)=\frac{\partial \varphi^2_{X^+}}{\partial t}(t_{1}(0,\T),0,0)\frac{\partial t_1}{\partial y}(0,\T)+\frac{\partial \varphi^2_{X^+}}{\partial y}(t_{1}(0,\T),0,0),$$ and 
$$\frac{d\la_\rho^s}{dy}(0)=\frac{\partial \varphi^2_{X^+}}{\partial t}(t_{2}(0,\rho),0,0)\frac{\partial t_2}{\partial y}(0,\rho)+\frac{\partial \varphi^2_{X^+}}{\partial y}(t_{2}(0,\rho),0,0).$$
Since, 
$$\frac{\partial t_1}{\partial y}(0,0)=-\frac{\frac{\partial \varphi^1_{X^+}}{\partial y}(0,0,0)}{\frac{\partial \varphi^1_{X^+}}{\partial t}(0,0,0)}, \quad \frac{\partial t_2}{\partial y}(0,0)=-\frac{\frac{\partial \varphi^1_{X^+}}{\partial y}(0,0,0)}{\frac{\partial \varphi^1_{X^+}}{\partial t}(0,0,0)},$$
and $$\frac{\partial \varphi^1_{X^+}}{\partial y}(0,0,0)=0,\, \frac{\partial \varphi^2_{X^+}}{\partial y}(0,0,0)=1,\, \kappa_{\T,\rho}^{s,u}=\frac{d T^{s,u}\Big|_{\sigma'_p}}{dy}(0)=\frac{d \la_{\rho,\T}^{s,u}}{dy}(0),$$
we get that
\begin{equation}\label{LT}\begin{array}{rl}\lim\limits_{\T \to 0} \kappa^u_\T=&\lim\limits_{\T \to 0}\frac{d\la_\T^u}{dy}(0)\\
 =& \frac{\partial \varphi^2_{X^+}}{\partial t}(t_{1}(0,0),0,0)\frac{\partial t_1}{\partial y}(0,0)+\frac{\partial \varphi^2_{X^+}}{\partial y}(t_{1}(0,0),0,0)\\ 
=&  1,
\end{array}\end{equation} and 
\begin{equation}\label{LR}\begin{array}{rl}\lim\limits_{\rho \to 0} \kappa^s_\rho=& \lim\limits_{\rho \to 0}\frac{d\la_\rho^s}{dy}(0) \\
 = & \frac{\partial \varphi^2_{X^+}}{\partial t}(t_{2}(0,0),0,0)\frac{\partial t_2}{\partial y}(0,0)+\frac{\partial \varphi^2_{X^+}}{\partial y}(t_2(0,0),0,0)\\ 
= &  1.
\end{array}\end{equation} 
The result follows from \eqref{LT} and \eqref{LR}.
\end{proof}

\subsection{Proof of Theorem \ref{tc1}}\label{sec:prooftc1}

		 First, from Theorem \ref{ta1} of Appendix B for $\T=x_\e$, there exist $\rho_0>0$ and constants $\beta<0$ and $c,r,q>0$ such that for every $\rho\in(\e^\la,\rho_0],$ $\la\in(0,\la^*),$ and $\e>0$ sufficiently small, the flow of $Z_{\e}^{\Phi}$ defines a map $U_{\e}$ between the transversal sections $\widehat V_{\rho,\la}^{\e}=\{-\rho\}\times [\e,y_{\rho,\la}^{\e}]\,\,\text{ and }\,\, \widetilde V_{x_\e}^{\e}=\{x_\e\}\times[y_{x_\e}^\e,y_{x_\e}^\e+r e^{-\frac{c}{\e^q}}]$ satisfying
\begin{equation}\label{Ue}
\begin{array}{cccl}
U_{\e}:& \widehat V_{\rho,\la}^{\e}& \longrightarrow& \widetilde V_{x_\e}^{\e}\\
&y&\longmapsto&y_{x_\e}^{\e}+\CO(e^{-\frac{c}{\e^q}}),
\end{array}
\end{equation}
where $y_{x_\e}^{\e} =\ov y_{x_\e}+\e+\mathcal{O}(\e^{2k\la^*})$ (see Figure \ref{figMAP1}).
\begin{figure}[h]
	\begin{center}
		\begin{overpic}[scale=0.9]{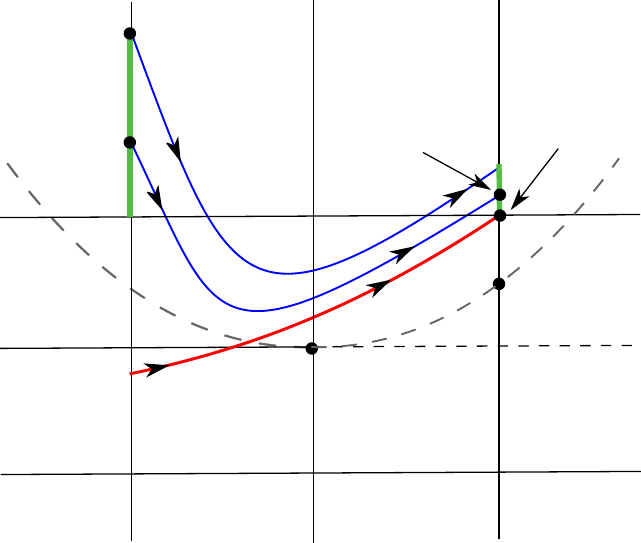}
		\put(11,70){$\widehat V_{\rho,\la}^{\e}$}
		\put(78,59){$\widetilde V_{x_\e}^{\e}$}
		\put(68,-3){$x=\T=x_\e$}
		\put(15,-3){$x=-\rho$}
		\put(98,32){$\Sigma$}
		\put(98,53){$y=\e$}
		\put(98,13){$y=-\e$}
		\put(87,64){$(x_\e,y_{x_\e}^{\e})$}
		\put(79,37){$\ov{y}_{x_\e}$}
		\put(55,61){$U_{\e}(y)$}
		\put(16,62){$y$}
		\put(12,79){$y^\e_{\rho,\la}$}
		\end{overpic}
	\end{center}	
	\bigskip	
	\caption{\tiny{Upper transition map $U_{\e}$ of the regularized system $Z^\Phi_\e.$ The dotted curve is the trajectory of $X^+$ passing through the visible regular-tangential singularity $p$ and the red curve is the well-known Fenichel manifold.}}
	\label{figMAP1}
	\end{figure}
		
		Now, notice that there exists $\e_0>0$ such that $\widetilde V_{x_\e}^{\e}\subset \tau_t^u,$ for all $\e\in[0,\e_0].$ 
		In this way, for $\e\in[0,\e_0],$ we define the function $\pi_\e(y)=D_\e\circ U_{\e}(y).$ Hence, from \eqref{Teps} and \eqref{Ue}, we get
		\begin{equation}\label{Piea}
		\begin{array}{lllll} \pi_{\e}(y) 
		& = &\displaystyle D_\e\Big(y^\e_{x_\e}+\mathcal{O}(e^{-c/\e^q})\Big)\\
		& = &\displaystyle D\Big(\overline{y}_{x_\e}+\e+\mathcal{O}\Big(\e^{2k\la^*}\Big)\Big)+\e S+\CO(\e^2)\\
		& = &\displaystyle \overline{y}_{-\rho}+r_{x_\e,\rho}\Big(\e+\mathcal{O}\Big(\e^{2k\la^*}\Big)\Big)+\mathcal{O}\Big(\e+\mathcal{O}\Big(\e^{2k\la^*}\Big)\Big)^2+\e S+\CO(\e^2)\\
		& = &\displaystyle \overline{y}_{-\rho}+\Big(r_{x_\e,\rho}+S\Big)\e+\mathcal{O}\Big(\e^{2k\la^*}\Big).\\
		\end{array}
		\end{equation}
Using \eqref{Piea} and \eqref{ye}, we have
		 $$\pi_\e(y)-y^\e_{\rho,\la}=\left(r_{x_\e,\rho}+S-1\right)\e+\mathcal{O}(\e\rho)-\beta \e^{2k\la}+\mathcal{O}(\e^{(2k+1)\la})+\mathcal{O}(\e^{1+\la})+\mathcal{O}\Big(\e^{2k\la^*}\Big).$$ Hence, we must study the limit $\lim\limits_{\rho,\e\rightarrow 0}\frac{\pi_\e(y)-y^\e_{\rho,\la}}{\e}$ in three distinct cases.

	 First, assume that $\la>\frac{1}{2k}.$ Then,
		$$\frac{\pi_\e(y)-y^\e_{\rho,\la}}{\e}=r_{x_\e,\rho}+S-1+\mathcal{O}(\rho)-\beta \e^{2k\la-1}+\mathcal{O}(\e^{(2k+1)\la-1})+\mathcal{O}(\e^{\la})+\mathcal{O}\Big(\e^{2k\la^*-1}\Big).$$ Thus, by Lemma \ref{rTR},
		\begin{equation}\label{eql1}	
		\lim\limits_{\rho,\e\rightarrow 0}\frac{\pi_\e(y)-y^\e_{\rho,\la}}{\e}=K+S-1.\end{equation}
		
		Now, suppose that $\la<\frac{1}{2k}.$ Then,
		$$\frac{\pi_\e(y)-y^\e_{\rho,\la}}{\e^{2k\la}}=(r_{x_\e,\rho}+S-1)\e^{1-2k\la}+\mathcal{O}(\e^{1-2k\la}\rho)-\beta+\mathcal{O}(\e^{\la})+\mathcal{O}\Big(\e^{2k(\la^*-\la)}\Big).$$ Hence, by Lemma \ref{rTR},
		\begin{equation}\label{eql2}	
		\lim\limits_{\rho,\e\rightarrow 0}\frac{\pi_\e(y)-y^\e_{\rho,\la}}{\e^{2k\la}}=-\beta>0.\end{equation}
		
		Finally,  assume that $\la=\frac{1}{2k}.$ Then, 
		$$\frac{\pi_\e(y)-y^\e_{\rho,\la}}{\e}=r_{x_\e,\rho}+S-1-\beta+\mathcal{O}(\rho)+\mathcal{O}(\e^{\la})+\mathcal{O}\Big(\e^{2k\la^*-1}\Big).$$ Thus, by Lemma \ref{rTR},
		\begin{equation}\label{eql3}	
		\lim\limits_{\rho,\e\rightarrow 0}\frac{\pi_\e(y)-y^\e_{\rho,\la}}{\e}=K+S-1-\beta.\end{equation}
	
		Now, we prove statement $(a)$ of Theorem \ref{tc1}. As $K+S-1>0,$ then all the above limits \eqref{eql1}, \eqref{eql2} and \eqref{eql3}, are strictly positive and, since $\e>0,$ there exists $\delta_0>0$ such that 
		$$0<\rho,\e<\delta_0\hspace{0.1cm}\Rightarrow\hspace{0.1cm}\pi_\e(y)-y^\e_{\rho,\la}>0.$$		
		Therefore, $\pi_\e([\e,y^\e_{\rho,\la}])\cap[\e,y^\e_{\rho,\la}]=\emptyset,$ for all $\e\in(0,\delta_0).$ This means that $\pi_\e$ has no fixed points in $[\e,y^\e_{\rho,\la}]$, i.e. the regularized system $Z_{\e}^{\Phi}$ does not admit limit cycles passing through the section $\widehat H_{\rho,\la}^{\e}.$
		
		 Now, we prove statement $(b)$ of Theorem \ref{tc1}. In this case, $\la>\frac{1}{2k}.$  As $K+S-1<0,$ then the limit \eqref{eql1} is strictly negative and, since $\e>0,$  there exists $\delta_0>0$ such that $0<\rho,\e<\delta_0\hspace{0.1cm}\Rightarrow\hspace{0.1cm}\pi_\e(y)-y^\e_{\rho,\la}<0.$ Consequently, $\pi_\e(y)<y^\e_{\rho,\la}.$ Moreover, from \eqref{Piea}, we get
		\begin{equation*}\label{eql4}	
		\lim\limits_{\e\rightarrow 0}\pi_\e(y)-\e=\ov{y}_{-\rho}>0,\end{equation*}
		for all $\rho>0.$ Since $\e>0,$  there exists $\delta_1>0$ such that 
		$0<\e<\delta_1\hspace{0.1cm}\Rightarrow\hspace{0.1cm}\pi_\e(y)-\e>0.$ Accordingly, $\pi_{\e}(y)>\e,$ for $\e>0$ sufficiently small. This means that $\pi_\e([\e,y^\e_{\rho,\la}])\subset[\e,y^\e_{\rho,\la}].$  By the {\it Brouwer Fixed Point Theorem}, we can conclude that $\pi_\e$ admits fixed points in $[\e,y^\e_{\rho,\la}]$, i.e. the regularized system $Z_{\e}^{\Phi}$ has periodic orbits passing through the section $\widehat H_{\rho,\la}^{\e}.$	
		
		In what follows, we shall prove the uniqueness of the fixed point in  $[\e,y^\e_{\rho,\la}].$ Indeed, expanding $D_\e$ in Taylor series around $y=y_{x_\e}^\e,$ we have that
		$$D_\e(y)=D_\e(y_{x_\e}^\e)+\frac{dD_\e}{dy}(y_{x_\e}^\e)(y-y_{x_\e}^\e)+\mathcal{O}((y-y_{x_\e}^\e)^2).$$ Hence,
		$$\begin{array}{lllll} \pi_{\e}(y) &= & D_\e(y_{x_\e}^\e+\mathcal{O}(e^{-c/\e^q}))\\
		& = &\displaystyle D_\e(y_{x_\e}^\e)+\frac{dD_\e}{dy}(y_{x_\e}^\e)\mathcal{O}(e^{-c/\e^q})+\mathcal{O}(e^{-2c/\e^q})\\
		& = &\displaystyle D_\e(y_{x_\e}^\e)+\mathcal{O}(e^{-c/\e^q}),\\
		\end{array}$$ and, therefore, $|\pi_\e(y_1)-\pi_\e(y_2)|=\mathcal{O}(e^{-c/\e^q}),$ for all $y_1,y_2\in[\e,y^\e_{\rho,\la}].$ Now, let $\nu_{\e}$ be the function given by
		$$\begin{array}{rcl} \nu_{\e}:  [\e,y^\e_{\rho,\la}]&\longrightarrow & [0,1]\\ 
		y &\longmapsto & \displaystyle\frac{y-\e}{y^\e_{\rho,\la}-\e}.\\ 
	\end{array}$$ Notice that $\nu_{\e}^{-1}(u)=(y^\e_{\rho,\la}-\e)u+\e.$ Thus, if $\widetilde{\pi}_\e(u)=\pi_\e\circ\nu_\e^{-1}(u),$ then $$|\widetilde{\pi}_\e(u_1)-\widetilde{\pi}_\e(u_2)|=\mathcal{O}(e^{-c/\e^q}),$$ for all $u_1,u_2\in[0,1].$ Fix $l\in(0,1),$ take $u_1,u_2\in[0,1],$ and define the function $\ell(\e):=(y^\e_{\rho,\la}-\e)l.$  There exists $\e(u_1,u_2)>0$ and a neighborhood $U(u_1,u_2)\subset[0,1]^2$ of $(u_1,u_2)$  such that 
$$|\widetilde{\pi}_\e(x)-\widetilde{\pi}_\e(y)|<\ell(\e)|x-y|,$$ for all $(x,y)\in U(u_1,u_2)$ and $\e\in(0,\e(u_1,u_2)).$ Since $\{U(u_1,u_2):\,(u_1,u_2)\in [0,1]^2\}$ is an open cover  of the compact set $[0,1]^2,$ there exists a finite sequence  $(u^i_1,u^i_2)\in[0,1]^2,$ $i=1,\ldots,s,$ for which $\{U^i\defeq U(u_1^i,u_2^i):\,i=1,\ldots,s\}$ still covers $[0,1]^2.$
Taking $\breve{\e}=\min\{\e(u^i_1,u^i_2):i=1,\ldots,s\},$ we get
$$|\widetilde{\pi}_\e(x)-\widetilde{\pi}_\e(y)|<\ell(\e)|x-y|,$$ for all $\e\in(0,\breve{\e})$ and $(x,y)\in[0,1]^2.$
Finally, since $\pi_{\e}(z)=\widetilde{\pi}_\e\circ\nu(z),$ we have that
$$\begin{array}{lllll} |\pi_\e(x)-\pi_\e(y)| &= & |\widetilde{\pi}_\e\circ\nu_\e(x)-\widetilde{\pi}_\e\circ\nu_\e(y)|\\
& < & \ell(\e)|\nu_\e(x)-\nu_\e(y)|\\
& = &\displaystyle \displaystyle\frac{\ell(\e)}{y^\e_{\rho,\la}-\e}|x-y|\\
& = &\displaystyle l|x-y|,\\
\end{array}$$ for all $\e\in(0,\breve{\e})$ and $x,y\in[\e,y_{\rho,\la}^{\e}].$ Therefore, $\pi_\e$ is a contraction for $\e>0$ small enough. From the \textit{Banach Fixed Point Theorem}, $\pi_\e$ admits a unique asymptotically stable fixed point for $\e>0$ small enough. Therefore, the regularized system $Z_{\e}^{\Phi}$ has a unique asymptotically stable limit cycle $\Gamma_{\e}$ passing through the section $\widehat H_{\rho,\la}^{\e},$ for $\e>0$ sufficiently small. Moreover, since $\pi_\e(y)-\ov{y}_{-\rho}=\mathcal{O}(\e)$ for all $y\in[\e,y_{\rho,\la}^{\e}]$ and $x_{\e}-\ov{x}^{+}_{\e}=\mathcal{O}(\e^{\frac{1}{2k}}),$ we get from differentiable dependency results on parameters and initial condition that $\Gamma_\e$ is $\e$-close to $\Gamma.$

\subsection{ Cases of uniqueness and nonexistence of limit cycles}\label{sec:nonexistence}

In the previous section, Theorem \ref{tc1} guaranteed the nonexistence and uniqueness of limit cycles in a specific compact set with nonempty interior. Nevertheless, it is not ensured, in general, the nonexistence and uniqueness of limit cycles converging to $\Gamma$, because this compact set degenerates into $\Gamma$ when $\e\rightarrow 0$. However, if we suppose, in addition, that $X^+$ has locally a unique isocline $x=\psi(y)$ of $2k-$multiplicity contacts with the straight lines $y=cte,$ then  we can guarantee the nonexistence and uniqueness of limit cycles converging to $\Gamma$. More precisely, consider the following proposition.

\begin{proposition}\label{proptc1}
Let $Z=(X^+,X^-)_{\Sigma}$ be a Filippov system and assume that  $Z$ has a $\Sigma-$polycycle $\Gamma$ of type $(a)$ satisfying $(a.1),$ $(a.2),$ and $(a.3).$ For $n\geq 2k-1,$ let $\Phi\in C^{n-1}_{ST},K,S$ be given as $\eqref{Phi},$ $\eqref{piK}$ and $\eqref{Teps}$ respectively and consider the regularized system $Z_{\e}^{\Phi}$ \eqref{regula}. If $K+S-1\neq 0$ and $X^+$ has locally a unique isocline $x=\psi(y)$ of $2k-$multiplicity contacts with the straight lines $y=\e,$ then the following statements hold.
\begin{enumerate}
	\item[(a)] If $K+S-1>0$ and $K>1$, then for $\e>0$ sufficiently small the regularized system $Z_{\e}^{\Phi}$ does not admit limit cycles converging to $\Gamma.$
	
		\item[(b)] If $K+S-1<0$ and $K<1$, then for $\e>0$ sufficiently small the regularized system $Z_{\e}^{\Phi}$ admits a unique limit cycle $\Gamma_{\e}$ converging to $\Gamma.$ Moreover, $\Gamma_{\e}$ is hyperbolic and asymptotically stable.
\end{enumerate}

\end{proposition}
\begin{remark}
In Proposition \ref{proptc1} we are assuming $(a.3),$ i.e. $X^-h(p)>0.$ If $X^-h(p)<0,$ Proposition \ref{proptc1} can be applied to $-Z$. Consequently, the unique limit cycle obtained for $Z,$ under the suitable assumptions, would be unstable. For more details see Remark \ref{remark1}.
\end{remark}
\begin{proof}
This proof will be split in 4 steps:

\textbf{Step 1.} First of all, we shall construct the first return map $\pi_\e$ of $Z_\e^\Phi,$ for $\e>0$. Since $X_1^+(p)>0,$ implies that there exists an open set $U,$ such that $X_1^+(x,y)\neq 0$, for all $(x,y)\in U.$ Take $\rho,\e>0$ small enough in order that the intersections of the trajectory of $Z_\e^\Phi$ starting at $(\psi(\e),\e)$ with the sections $\{x=x_\e\}$ and $\{x=-\rho\}$ are contained in $U$, namely $q_\e^{u}=(x_\e,\ov{y}^\e_{x_\e})$ and $q_\e^{s}=(-\rho,\ov{y}^\e_{-\rho})$, respectively. Thus, there exist $\delta_\e^{u,s}$ positive numbers, such that       
\[\begin{array}{l}
\tau^{u}_{t,\e}=\{(x_\e,y):y\in(\ov{y}^\e_{x_\e}-\delta_\e^{u},\ov{y}^\e_{x_\e}+\delta_\e^{u})\},\\
\tau^{s}_{t,\e}=\{(-\rho,y):y\in(\ov{y}^\e_{-\rho}-\delta_\e^{s},\ov{y}^\e_{-\rho}+\delta_\e^{s})\},\\
\end{array}\] are transversal sections of $X^+$. Moreover, notice that the arc-orbits connecting $(\psi(\e),\e)$ with $q_\e^u$ and $(\psi(\e),\e)$ with $q_\e^s$ are contained in $\Sigma^+$ (see Figure \ref{mirrormap1}).
Thus, by \cite[Theorem A]{AndGomNov19} we know that there exist the transition maps $R_\e^{u}:\sigma^+_{p,\e}:=[\psi(\e),x_\e]\times\{\e\}\longrightarrow\tau^{u}_{t,\e}$ and $R_\e^{s}:\sigma^-_{p,\e}:=[-\rho,\psi(\e)]\times\{\e\}\longrightarrow\tau^{s}_{t,\e}$ satisfying
\begin{equation}\label{newTsu2}
\begin{split}
R_\e^{u}(x)&=\ov{y}^\e_{x_\e}+\varsigma_{x_\e,\e}^{u}(x-\psi(\e))^{2k}+\CO\left((x-\psi(\e))^{2k+1}\right),\\
R_\e^{s}(x)&=\ov{y}^\e_{-\rho}+\varsigma_{\rho,\e}^{s}(x-\psi(\e))^{2k}+\CO\left((x-\psi(\e))^{2k+1}\right),\\
\end{split}
\end{equation}
where $\sgn(\varsigma_{x_\e,\e}^{u})=-\sgn((X^{+})^{2k}h(\psi(\e)))=\sgn(\varsigma_{\rho,\e}^{s})$, i.e. $\varsigma_{x_\e,\e}^u,\varsigma_{\rho,\e}^s<0.$ Furthermore, from \textit{Implicit Function Theorem}, it is easy to see  that
$$(R_\e^{s})^{-1}(y)=\psi(\e)-\sqrt[2k]{\frac{1}{-\varsigma_{\rho,\e}^{s}}}(\ov{y}^\e_{-\rho}-y)^{\frac{1}{2k}}+\CO\left((\ov{y}^\e_{-\rho}-y)^{1+\frac{1}{2k}}\right).$$ 
Now, we know that there exists a diffeomorphism 
$P^e_\e:\tau^u_{t,\e}\longrightarrow\tau^s_{t,\e}$ defined as
$$P^e_\e(y)=\ov{y}^\e_{-\rho}+r^\e_{x_\e,\rho}(y-\ov{y}^\e_{x_\e})+\CO((y-\ov{y}^\e_{x_\e})^2)+\CO(\e).$$
\begin{figure}[h]
	\begin{center}
    \begin{overpic}[scale=0.9]{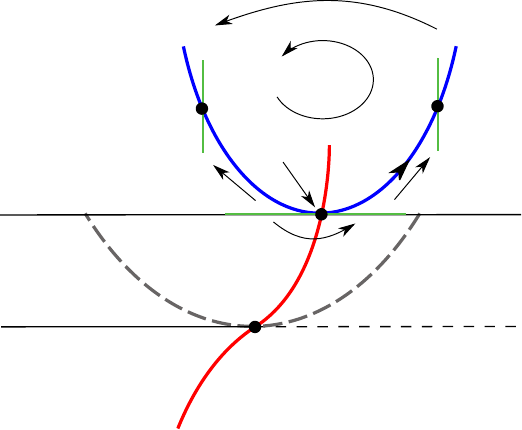}
		\put(50,53){\scriptsize $\psi(\e)$ \par}
		\put(50,16){$p$}
		\put(37,44){$R_\e^s$}
		\put(80,44){$R_\e^u$}
    \put(31,55){$\tau^s_{t,\e}$}
		 \put(85,55){$\tau^u_{t,\e}$}
		  \put(78,64){\scriptsize $q^u_{\e}$ \par}
		  \put(40,64){\scriptsize $q^s_{\e}$ \par}
		  \put(45,38){\scriptsize $\sigma^-_{p,\e}$ \par}
		  \put(71,38){\scriptsize $\sigma^+_{p,\e}$ \par}
		\put(101,22){$\Sigma$}
		\put(101,44){$y=\e$}
		\put(60,66){$\pi_\e$}
		\put(60,85){$P_\e^e$}
		\put(61,33){$\rho_\e$}
		\put(41,7){$x=\psi(y)$}
		\end{overpic}
		\caption{\tiny{The first return map $\pi_\e$ of $Z_\e^\Phi.$}}
	\label{mirrormap1}
	\end{center}
	\end{figure}
Consequently, we get the first return map $\pi_\e:\tau^s_{t,\e}\longrightarrow\tau^s_{t,\e}$ given by
\begin{equation}\label{pireg1}\begin{array}{rcl}
\pi_\e(y)&=&P^e_\e\circ R_\e^u\circ\rho_\e\circ(R_\e^s)^{-1}(y)\\
&=&\ov{y}^\e_{-\rho}-\dfrac{r^\e_{x_\e,\rho}\varsigma_{x_\e,\e}^{u}}{\varsigma_{\rho,\e}^{s}}(\ov{y}^\e_{-\rho}-y)+\CO((\ov{y}^\e_{-\rho}-y)^p)+\CO(\e)
\end{array}\end{equation}
for some $p>1,$ where $\rho_\e:\sigma^-_{p,\e}\longrightarrow\sigma^+_{p,\e}$ is defined as
$$\rho_\e(x)=-x+2\psi(\e)+\CO\left((x-\psi(\e))^2,\e(x-\psi(\e))\right).$$ 
In \cite{NovRon2019}, $\rho_\e$ is called \textit{mirror map} associated with $Z_{\e}^{\Phi}$ at $\psi(\e)$ (see Figure \ref{mirrormap1}). For more details about the mirror map see, for instant, section 8.1 in \cite{NovRon2019}.	

\textbf{Step 2.} In this step we show that $\lim_{\e,\rho \to 0}\varsigma_{x_\e,\e}^{u}/\varsigma_{\rho,\e}^{s}=1.$ For that, notice that for $\rho,\T>0$ sufficiently small:
\begin{itemize}
\item $\lim\limits_{\e \to 0}\ov{y}^\e_{-\rho}=\ov{y}_{-\rho},$
\item $\lim\limits_{\e \to 0}\ov{y}^\e_{\T}=\ov{y}_{\T},$
\end{itemize}
where $\ov{y}_{-\rho}$ and $\ov{y}_{\T}$ are given in Lemma \ref{y0}. We remark that $R^s_\e(-\rho)=\e$ and $R^u_\e(\T)=\e$. Hence, by \eqref{newTsu2} we get the following equations 
\[\begin{split}
\ov{y}^\e_{-\rho}+\varsigma_{\rho,\e}^{s}(-\rho-\psi(\e))^{2k}+\CO\left((-\rho-\psi(\e))^{2k+1}\right)&=\e,\\
\ov{y}^\e_{\T}+\varsigma_{\T,\e}^{u}(\T-\psi(\e))^{2k}+\CO\left((\T-\psi(\e))^{2k+1}\right)&=\e.\\
\end{split}\]
Making $\e$ tend to 0 and using that $\lim\limits_{\e \to 0}\psi(\e)=0,$ we obtain that
$$
\ov{y}_{-\rho}+\varsigma_{\rho,0}^s\rho^{2k}+\CO(\rho^{2k+1})=0\quad\text{and}\quad
\ov{y}_\T+\varsigma_{\T,0}^u\T^{2k}+\CO(\T^{2k+1})=0.$$

By Lemma \ref{y0}, we get
\[\begin{split}
\frac{\al\rho^{2k}}{2k}+\CO(\rho^{2k+1})+\varsigma_{\rho,0}^s\rho^{2k}+\CO(\rho^{2k+1})&=0,\\
\frac{\al\T^{2k}}{2k}+\CO(\T^{2k+1})+\varsigma_{\T,0}^u\T^{2k}+\CO(\T^{2k+1})&=0,\\
\end{split}\]
i.e. 
\[
\frac{\al}{2k}+\varsigma_{\rho,0}^s+\CO(\rho)=0\quad\text{and}\quad
\frac{\al}{2k}+\varsigma_{\T,0}^u+\CO(\T)=0.
\]
Consequently, $\varsigma_{\rho,0}^s,\varsigma_{\T,0}^u$ tend to $-\frac{\alpha}{2k}$, when $\rho,\T$ tend to 0. Thus,
\[
\lim\limits_{\e,\rho  \to 0}\varsigma_{\rho,\e}^s=-\frac{\alpha}{2k}\quad\text{and}\quad
\lim\limits_{\e,\T \to 0}\varsigma_{\T,\e}^u=-\frac{\alpha}{2k}.
\]
Therefore, taking $\T=x_\e$ we conclude that $\lim_{\e,\rho \to 0}\varsigma_{x_\e,\e}^{u}/\varsigma_{\rho,\e}^{s}=1.$

\textbf{Step 3.} Now, if $\Gamma_\e$ is a limit cycle of the regularized system $Z_{\e}^{\Phi}$ converging to $\Gamma$ (i.e. there exists a fixed point $(-\rho,y_\e^\rho)\in\tau^s_{t,\e}$ of $\pi_\e$ such that $\lim\limits_{\e \to 0}y_\e^\rho=\ov{y}_{-\rho}$), then by \eqref{pireg1} we have
$$\frac{d\pi_\e}{dy}(y)=\dfrac{r^\e_{x_\e,\rho}\varsigma_{x_\e,\e}^{u}}{\varsigma_{\rho,\e}^{s}}+\CO\left((\ov{y}^\e_{-\rho}-y)^{p-1}\right).$$
Notice that $\lim\limits_{\e \to 0}r_{\T,\rho}^\e=r_{\T,\rho}$ when $\T$ is small enough. Hence, using Lemma \ref{rTR} and step 2, we get
$$\lim\limits_{\e,\rho \to 0}\frac{d\pi_\e}{dy}(y_\e^\rho)=K.$$
Therefore, if $K>1$ (resp. $K<1$), then $\Gamma_\e$ is hyperbolic and unstable (resp. asymptotically stable), for $\e>0$ small enough.

\begin{figure}[h]
	\begin{center}
    \begin{overpic}[scale=0.7]{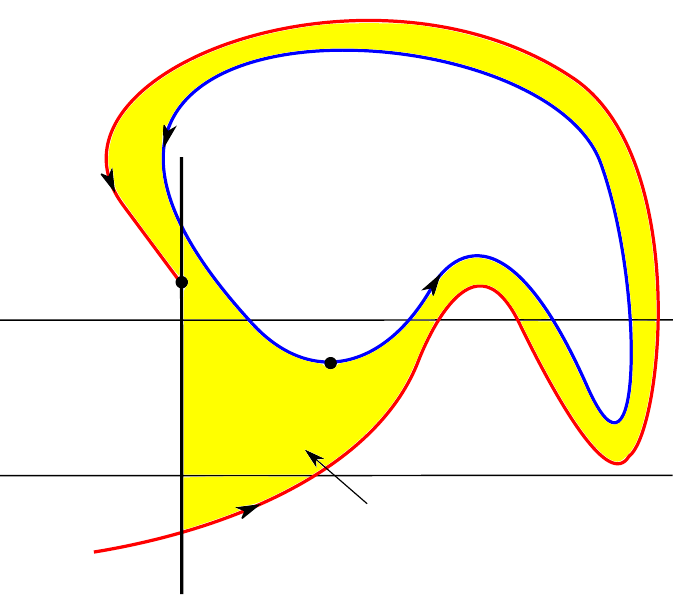}
		\put(5,8){$S_{a,\e}$}
		\put(21,-1){$x=-\rho$}
		\put(45,77){$\Gamma_\e$}
		\put(101,21){$\Sigma$}
		\put(101,45){$y=\e$}
		\put(55,13){$\mathcal{B}_\e$}
		\end{overpic}
		\caption{\tiny{The region $\mathcal{B}_\e$.}}
	\label{regiongsr1}
	\end{center}
	\end{figure}
	
\textbf{Step 4.} Now, we are ready to prove this proposition. The proof of the first statement is by contradiction. Assume that there exists a limit cycle $\Gamma_\e$ of $Z_{\e}^{\Phi}$ such that $\Gamma_\e$ converging to $\Gamma,$ for $\e>0$ sufficiently small. Let $\mathcal{B}_\e$ be the region delimited by the curves $x=-\rho,$ the limit cycle $\Gamma_\e$ and the Fenichel manifold $S_{a,\e}$ associated with $Z_{\e}^{\Phi},$ (see Figure \ref{regiongsr1}). Since $\Gamma_\e$ converges to the regular orbit $\Gamma$ then $\mathcal{B}_\e$ has no singular points. Moreover, it is easy to see that  $\mathcal{B}_\e$ is positively invariant compact set, for $\e>0$ small enough. For $\e>0$ choose $q_\e\in\mathcal{B}_\e,$ from the \textit{Poincar\'{e}--Bendixson Theorem} $\omega(q_\e)\subset \mathcal{B}_\e$ is a limit cycle of $Z_{\e}^{\Phi}$ which is not unstable, but this contradicts step 3.

Now, we prove the second statement. By Theorem \ref{tc1}, for $\e>0$ sufficiently small, we know that $Z_{\e}^{\Phi}$ has a asymptotically stable limit cycle $\Gamma_\e$ converging to $\Gamma.$ In addition, from step 3 we have that $\Gamma_\e$ is hyperbolic. Finally, we claim that $\Gamma_\e$ is the unique limit cycle with these properties. In fact, assume that there exists another limit cycle $\widetilde{\Gamma_\e}$ converging to $\Gamma$, hyperbolic and asymptotically stable. Now, let $\mathcal{R}_\e$ be the region delimited by the limit cycles $\Gamma_\e$ and $\widetilde{\Gamma_\e}.$  Notice that $\mathcal{R}_\e$ is negatively invariant compact set. Furthermore, since $\Gamma_\e$ and $\widetilde{\Gamma_\e}$ converges to the regular orbit $\Gamma,$ we can conclude that $\mathcal{R}_\e$ has no singular points for $\e>0$ small enough. For $\e>0$ choose $q_\e\in\mathcal{R}_\e,$ from the \textit{Poincar\'{e}--Bendixson Theorem} we can conclude that $\alpha(q_\e)\subset\mathcal{R}_\e$ is a limit cycle of $Z_{\e}^{\Phi}$ that is not asymptotically stable, but this contradicts step 3.
\end{proof}

\section{Regularization of $\Sigma-$Polycycles of type $(b)$}\label{sec:polycyclesb}
In this section, we shall state and prove the second main result of this paper, which in particular establishes sufficient conditions under which the regularized vector field  $Z_{\e}^{\Phi}$ has a limit cycle $\Gamma_{\e}$ converging to a $\Sigma-$polycycle of type $(b).$ For this, suppose that a Filippov system $Z=(X^+,X^-)$ has a $\Sigma-$polycycle $\Gamma$ of type $(b)$.
Without loss of generality, assume that: 
\begin{enumerate}
\item[(b.1)]$X_1^+(p)>0$;
 \item[(b.2)]  the trajectory of $Z$ through $p$ crosses $\Sigma$ transversally $m-$times 
at $q_1,\cdots, q_m$, satisfying that for each $i = 0,\cdots, m$, there exists $t_i>0$ such that $\varphi_Z(t_i,q_i)=q_{i+1}$, where $q_{m+1}=p=q_0$. Moreover, $\Gamma\cap\Sigma =\{q_1,\cdots, q_m, p\}.$ 
\end{enumerate}

We shall also assume that 
\begin{enumerate}
\item[(b.3)] $W^s_{\ptt}(p)\cup W^u_t(p)\subset \Gamma,$ i.e. $X^-h(p)>0.$
\end{enumerate}
The case $W^u_{\ptt}(p)\cup W^s_t(p)\subset \Gamma$ is obtained from the previous case multiplying the vector field $Z$ by -1 (see Remark \ref{remark2} bellow). 

Since $X^-h(p)>0,$ by the \textit{Tubular Flow Theorem} for $X^-$ at $p=(0,0),$ there exist an open set $U$ and a local $\mathcal{C}^{2k}$ diffeomorphism $\psi$ defined on $U$ such that $\widetilde X^-=\psi_*X^-=(0,1).$  Clearly, the transformed vector field $\widetilde X^+=\psi_*X^+$ still has a visible $2k$-multiplicity contact with $\Sigma$ at $p=(0,0)$ and $\psi(\Sigma)=\Sigma.$ Thus, without loss of generality, we can assume that there exists a neighborhood $U\subset\R^2$ of $p=(0,0)$ such that $X^-\big|_U=(0,1).$

Notice that assumption $(b.2)$ above guarantees the existence of an exterior map $D$ associated to $Z.$ In what follows, we characterize such a map. 
Since $X_1^+(p), X_2^-(p)>0,$ then $X_1^+(x,y), X_2^-(x,y)> 0,$ for all $(x,y)\in U$ (taken smaller if necessary). Take $\e,\T>0$ small enough in order that the points $q^{u}=(\T,\ov{y}_\T)\in W^{u}_t(p)$ and $q^{s}=(0,-\e)\in W^{s}_{\ptt}(p)$ are contained in $U$. Then there exist positive numbers $\delta^{u,s}$ such that
\begin{equation}\label{sectrab}
\begin{array}{l}
\tau^{u}_t=\{(\T,y):y\in(\ov{y}_\T-\delta^{u},\ov{y}_\T+\delta^{u})\} \,\text{and}\\
\tau^{s}_{\ptt}=\{(x,-\e):x\in(-\delta^{s},\delta^{s})\}\\
\end{array}\end{equation} are transversal sections of $X^+$ and $X^{-}$, respectively. In addition, $\sigma_p=[0,\T]\times\{0\}$ is a transversal section of $X^{-}$. Moreover, by the \textit{Tubular Flow Theorem} there exist the $C^{2k}-$diffeomorphism $T^{s}:\sigma_p\longrightarrow\tau^{s}_{\ptt}$ and $D:\tau^{u}_t\longrightarrow\tau^{s}_{\ptt}$ such that $T^{s}(p)=q^{s}$ and $D(q^u)=q^s$ (see Figure \ref{figtan2}). Thus, expanding $D$ around $y=\ov{y}_{\T}$, we get
\begin{equation}\label{Dii}
D(y)=r_{\T,\e}(y-\ov{y}_{\T})+\CO((y-\ov{y}_{\T})^2),
\end{equation}
where $r_{\T,\e}=\frac{d D}{dy}(\ov{y}_{\T}).$
 \begin{figure}[h]
	\begin{center}
    \begin{overpic}[scale=0.6]{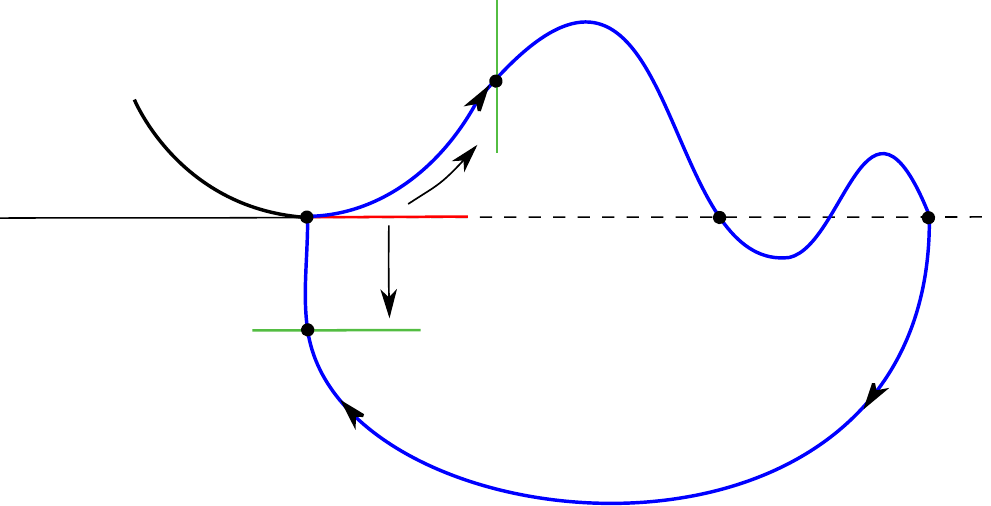}
		\put(28,26){$p$}
		\put(70,26){$q_1$}
		\put(95,26){$q_2$}
		\put(42,21){$T^s$}
		\put(47,30){$T^u$}
    \put(21,16){$\tau^s_{\ptt}$}
		 \put(48,54){$\tau^u_{t}$}
		  \put(47,44){\scriptsize $q^u$ \par}
		  \put(33,15){\scriptsize $q^s$ \par}
		\put(99,31){$\Sigma$}
		\put(91,36){$\Gamma$}
		\end{overpic}
		\caption{\tiny{$\Sigma-$polycycle $\Gamma$ satisfying $(b.1),$ $(b.2),$ and $(b.3)$.}}
	\label{figtan2}
	\end{center}
	\end{figure}

Now, by the Definition \ref{defpoly}, we know that there exists a first return map $\pi_\Gamma$ defined, at least, in one side of $\Gamma.$  In what follows, we shall see that there exist positive constants $\T$ and $K$ such that the first return map $\pi_\Gamma:\sigma_p\rightarrow\sigma_p$ is given by 
	\begin{equation}\label{piK2}
	\pi_\Gamma(x)=Kx^{2k}+\CO(x^{2k+1}), \,\text{where $K>0.$}
	\end{equation}
Indeed, expanding $T^{s}$ around $x=0$, we have 
\begin{equation}\label{Tsd2}
T^{s}(x)=\kappa_{\e}^{s}x+\CO(x^{2}),
\end{equation}
with $\kappa_{\e}^{s}=\frac{d T^{s}}{dx}(0).$ In addition, using \eqref{Dii} and the \textit{Implicit Function Theorem}, we get
\begin{equation}\label{Dinv2}
D^{-1}(x)=\ov{y}_\T+\frac{1}{r_{\T,\e}}x+\CO(x^2).
\end{equation}
Since $X^+$ and $X^-$ are planar vector fields, the uniqueness of solutions implies that  $r_{\T,\e}<0$ and $\kappa_{\e}^{s}>0$  for all $\T,\e>0$ sufficiently small.

We remark that the arc-orbit connecting $p$ with $q^u$ is contained in $\Sigma^+$ (see Figure \ref{figtan2}). Then, from \cite[Theorem A]{AndGomNov19} we know that there exists a transition map $T^{u}:\sigma_p\longrightarrow\tau^{u}_{t}$ defined as
\begin{equation}\label{Tsu2}
\begin{split}
T^{u}(x)&=\ov{y}_\T+\kappa_{\T}^{u}x^{2k}+\CO(x^{2k+1}),\\
\end{split}
\end{equation}
where $\sgn(\kappa_{\T}^u)=-\sgn((X^{+})^{2k}h(p))$, i.e. $\kappa_{\T}^u<0$.

Thus, we can define the first return map $\pi_\Gamma:\sigma_p\longrightarrow\sigma_p$ by
	$$\pi_\Gamma(x)=(D^{-1}\circ T^s)^{-1}\circ T^u(x).$$ 
From \eqref{Tsd2} and \eqref{Dinv2}, we have that $D^{-1}\circ T^s(x)=\ov{y}_\T+\frac{\kappa_{\e}^s}{r_{\T,\e}}x+\CO(x^{2}).$ Hence, using the \textit{Implicit Function Theorem} we conclude that
	$$(D^{-1}\circ T^s)^{-1}(x)=\frac{r_{\T,\e}}{\kappa_\e^s}(x-\ov{y}_\T)+\CO((x-\ov{y}_\T)^2).$$
 Thus, 
\begin{equation}\label{pinK2}
\pi_\Gamma(x)=\frac{r_{\T,\e}\kappa_\T^u}{\kappa_\e^s}x^{2k}+\CO(x^{2k+1}).
\end{equation}
Accordingly, taking $K:=\frac{r_{\T,\e}\kappa_\T^u}{\kappa_\e^s}>0,$ we get \eqref{piK2}.
\begin{remark}\label{Gstable}
Notice that for $x\in[0,\T]$ we have that $\pi_\Gamma(x)<x,$ for some small $\T>0$. This means that $\Gamma$ is always asymptotically stable provided that $(b.3)$ is satisfied, i.e. $X^-h(p)>0.$ If $X^-h(p)<0,$ $\Gamma$ would be unstable.
\end{remark}
Now, since $D$ is a diffeomorphism induced by a regular orbit, we can easily see that the regularized system  $Z_{\e}^{\Phi}$ also admits an exterior map given by 
\begin{equation}\label{Tepsb}
D_\e(y)=D(y)+\CO(\e).
\end{equation}
In what follows, we shall state the second main theorem of this article that will be proven in Section \ref{sec:prooftc2}.

\begin{mtheorem}\label{tc2}
Consider a Filippov system $Z=(X^+,X^-)_{\Sigma}$ and assume that $Z$ has a $\Sigma-$polycycle $\Gamma$ of type $(b)$ satisfying $(b.1),$ $(b.2),$ and $(b.3).$ For $n\geq 2k-1,$ let $\Phi\in C^{n-1}_{ST}$ be given as $\eqref{Phi}$ and consider the regularized system $Z_{\e}^{\Phi}$ \eqref{regula}. Then the regularized system $Z_{\e}^{\Phi}$ \eqref{regula} admits at least a limit cycle $\Gamma_{\e},$ for $\e>0$ sufficiently small. Moreover, $\Gamma_{\e}$ converges to $\Gamma.$
\end{mtheorem}
\begin{remark}\label{remark2}
In Theorem \ref{tc2} we are assuming $(b.3),$ i.e. $W^u_{t}(p)\cup W^s_{\ptt}(p)\subset \Gamma.$ If  $W^u_{\ptt}(p)\cup W^s_t(p)\subset \Gamma,$ Theorem \ref{tc2} can be applied to $-Z$. Consequently, we get a limit cycle for $Z.$ For more details see Remark \ref{remark1}.
\end{remark}

In order to prove this theorem, we shall establish the relationship between the derivative of the first return map $K$ and the derivative of exterior map $r_{\T,\e}$ as follows.

\begin{lemma}\label{lims} Consider $\kappa_{\T}^u$ and $r_{\T,\e}$ given as in \eqref{Tsu2} and \eqref{Dii}, respectively. Then,
\begin{enumerate}
\item[i)] $\lim\limits_{\T \to 0}\kappa^u_{\T}=-\frac{\al}{2k}.$
\item [ii)] $\lim\limits_{\T,\e \to 0}r_{\T,\e}=-\frac{2kK}{\al}.$
\end{enumerate}
\end{lemma}
\begin{proof}
First, we prove the statement $i)$. Indeed, using \eqref{Tsu2} and that $T^u(\T)=0$ for all $\T>0$ small enough, we have $\ov{y}_\T+\kappa_\T^u\T^{2k}+\CO(\T^{2k+1})=0.$ By Lemma \ref{y0} we know that $\ov{y}_\T=\frac{\al\T^{2k}}{2k}+\CO(\T^{2k+1})$, thus
$$\frac{\al\T^{2k}}{2k}+\CO(\T^{2k+1})+\kappa_\T^u\T^{2k}+\CO(\T^{2k+1})=0,$$ i.e. 
$$\frac{\al}{2k}+\kappa_\T^u+\CO(\T)=0,$$ consequently, when $\T$ tends to 0 we conclude the result.

Now, we shall prove statement $ii)$. For this, notice that using equations \eqref{piK2} and \eqref{pinK2} we get for $\T,\e>0$ small enough that
$r_{\T,\e}=\frac{K\kappa_\e^s}{\kappa_\T^u}.$
To prove statement $ii)$, we just need to prove that $\lim\limits_{\e \to 0}\kappa_\e^s=1,$ because in this case we have 
\[\begin{array}{rl}
\displaystyle\lim\limits_{\T,\e \to 0}r_{\T,\e}=&\lim\limits_{\T,\e \to 0}\dfrac{K\kappa_\e^s}{\kappa_\T^u}\\
=&\displaystyle K\dfrac{\lim\limits_{\e \to 0}\kappa_\e^s}{\lim\limits_{\T \to 0}\kappa_\T^u}\\
=&\displaystyle-\dfrac{2kK}{\al}.
\end{array}\]
where we have used item $i)$. In what follows, we shall prove that the flow of $X^-$ induces $\mathcal{C}^{2k}$ map,
$\la_\e^s:\sigma'_p\subset\sigma_p\longrightarrow\tau^s_{\ptt},$
between the transversal sections defined in \eqref{sectrab} and satisfying $\la_\e^s(0)=q^s$. Indeed, consider the function  
$$\mu(t,x,\e)=\varphi^2_{X^-}(t,x,0)+\e,\hspace{0.2cm}\text{for} \hspace{0.2cm} (x,0)\in \sigma_p,\hspace{0.2cm} t\in I_{(0,y)},$$ where $\varphi_{X^-}$ is the flow of $X^-$ and $I_{(x,0)}$ is the maximal interval of existence of $t\mapsto\varphi_{X^-}(t,x,0)$. Since, 
$$\mu(0,0,0)=0\quad\text{and}\quad\frac{\partial \mu}{\partial t}(0,0,0)=\frac{\partial \varphi^2_{X^-}}{\partial t}(0,0,0)=X_2^-(p)\neq 0,$$ 
by the \textit{Implicit Function Theorem} there exist $\T_0>0$ and smooth function $t(x,\e)$ with $(x,0)\in\sigma'_p:=[0,\T_0)\times\{0\}\subset\sigma_p$ and $\T,\e>0$ sufficiently small such that 
$$t(0,0)=0 \quad\text{and}\quad\mu(t(x,\e),x,\e)=0,$$ 
i.e. $\varphi^2_{X^-}(t(x,\e),x,0)=-\e.$ Thus, we can define the function
\[
\la_\e^s(x)=\varphi^1_{X^-}(t(x,\e),x,0).
\]
Notice that 
$$\frac{d\la_\e^s}{dx}(0)=\frac{\partial \varphi^1_{X^-}}{\partial t}(t(0,\e),0,0)\frac{\partial t}{\partial x}(0,\e)+\frac{\partial \varphi^1_{X^-}}{\partial x}(t(0,\e),0,0).$$
Since, 
$$\frac{\partial t}{\partial x}(0,0)=-\frac{\frac{\partial \varphi^2_{X^-}}{\partial x}(0,0,0)}{\frac{\partial \varphi^2_{X^-}}{\partial t}(0,0,0)},$$
and $$\frac{\partial \varphi^1_{X^-}}{\partial x}(0,0,0)=1,\, \frac{\partial \varphi^2_{X^-}}{\partial x}(0,0,0)=0,\, \kappa_{\e}^{s}=\frac{d T^{u}\Big|_{\sigma'_p}}{dx}(0)=\frac{d \la_\e^{s}}{dx}(0),$$
we get that
\begin{equation}\label{LR2}\begin{array}{rl}\lim\limits_{\e \to 0} \kappa^s_\e=& \lim\limits_{\e \to 0}\frac{d\la_\e^s}{dx}(0) \\
 = & \frac{\partial \varphi^1_{X^-}}{\partial t}(t(0,0),0,0)\frac{\partial t}{\partial x}(0,0)+\frac{\partial \varphi^1_{X^-}}{\partial x}(t(0,0),0,0)\\  
= &  1.
\end{array}\end{equation} 
The result follows from \eqref{LR2}.
\end{proof}

\subsection{Proof of Theorem \ref{tc2}}\label{sec:prooftc2}

Let $\U$ be an open set such that $\Gamma\subset \U.$ Since $\Gamma$ is a $\Sigma-$polycycle of $Z$, then it is easy to see that there exists an open set $V\subset \U$ such that $V$ has no singular points of $Z_{\e}^{\Phi}$. Taking $(\delta,0)\in V,$ we get that $(\delta,\e)\in V,$ for $\e>0$ small enough.

Now, define the map $\widetilde{D}_\e(x)=\pi_\Gamma(x)+\CO(\e),$ for all $x\in I_\delta^\e,$ where $I_\delta^\e$ is a neighborhood of $\delta.$ By Remark \ref{Gstable} we know that $\pi_\Gamma(\delta)<\delta,$ then we can conclude that $\widetilde{D}_\e(\delta)<\delta,$ for $\e>0$ sufficiently small. 

By Theorem \ref{tb1} of Appendix B for $\T=\ov{x}^{+}_{\e}$,  
there exist $\rho_0>0,$ and constants $c,r,q>0$ such that for every $\rho\in(\e^\la,\rho_0],$ 
$\la\in(0,\la^*),$ and $\e>0$ sufficiently small, the flow of $Z_{\e}^{\Phi}$ defines a map $L_{\e}$  between the transversal sections $\widecheck H_{\rho,\la}^{\e}=[-\rho,-\e^{\la}]\times\{-\e\}\,\,\text{ and }\,\, \widecheck V_{\ov{x}^{+}_{\e}}^{\e}=\{\ov{x}^{+}_{\e}\}\times[y_{\ov{x}^{+}_{\e}}^\e-r e^{-\frac{c}{\e^q}},y_{\ov{x}^{+}_{\e}}^\e]$ satisfying
\begin{equation}\label{le}
\begin{array}{cccl}
L_{\e}:& \widecheck H_{\rho,\la}^{\e}& \longrightarrow& \widecheck V_{\ov{x}^{+}_{\e}}^{\e}\\
&x&\longmapsto&y_{\ov{x}^{+}_{\e}}^{\e}+\CO(e^{-\frac{c}{\e^q}}),
\end{array}
\end{equation}
where  $y_{\ov{x}^{+}_{\e}}^\e=\ov{y}_{\ov{x}^{+}_{\e}}+\e+\CO(\e^{1+\frac{1}{2k}})+\CO(\e^{1+\la^*})+\CO(\e^{2k\la^*})$ (see Figure \ref{figMAP11}).
\begin{figure}[H]
	\begin{center}
		\begin{overpic}[scale=0.5]{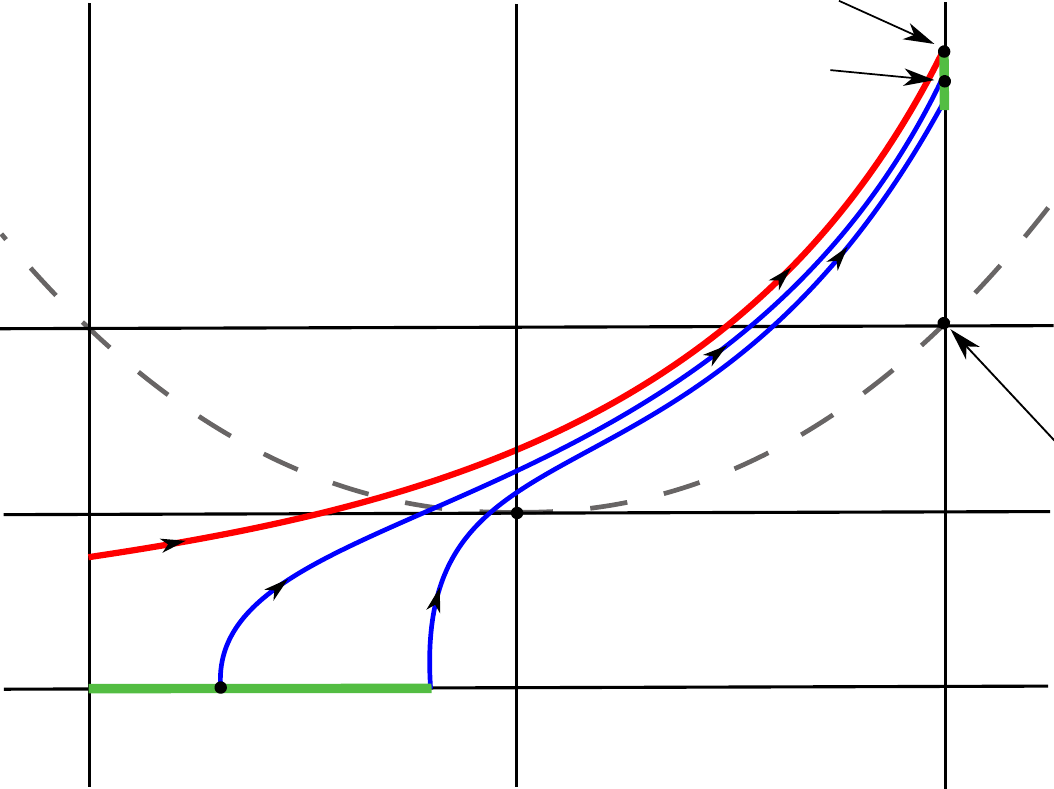}
\put(18,12){$x$}
	  \put(18,3){$\widecheck H_{\rho,\la}^{\e}$}
		\put(93,65){$\widecheck V_{\T}^{\e}$}
		\put(67,66){$L_\e(x)$}
		\put(72,75){$y_{\ov{x}^{+}_{\e}}^\e$}
		\put(101,31){\scriptsize $(\ov{x}^{+}_{\e},\ov{y}_{\ov{x}^{+}_{\e}})$ \par}
		\put(79,-4){$x=\T=\ov{x}^{+}_{\e}$}
		\put(2,-4){$x=-\rho$}
		\put(96,28){$\Sigma$}
		\put(96,46){$y=\e$}
		\put(96,12){$y=-\e$}
		\end{overpic}
	\end{center}
	
	\bigskip
	
	\caption{\tiny{Lower Transition Map $L_{\e}$ of the regularized system $Z^\Phi_\e.$ The dotted curve is the trajectory of $X^+$ passing through the visible regular-tangential singularity of multiplicity $2k$ $p=(0,0).$ The red curve is the well-known Fenichel manifold.}}
	\label{figMAP11}
	\end{figure}
Since $\widecheck V_{\ov{x}^{+}_{\e}}^{\e}\subset\tau_t^u$, for $\e>0$ sufficiently small, then we can consider the map $D_\e\circ L_\e:\widecheck H_{\rho,\la}^{\e}\longrightarrow\tau^{s}_{\ptt}.$ 
 We claim that $D_\e\circ L_\e(-\e^\la)>-\e^\la.$ Indeed, from \eqref{Tepsb} and \eqref{le}, we get
		\begin{equation*}
		\begin{array}{lllll} D_{\e}(L_\e(-\e^\la)) 
		& = &\displaystyle D_\e\Big(y^\e_{\ov{x}^{+}_{\e}}+\mathcal{O}(e^{-c/\e^q})\Big)\\
		& = &\displaystyle D\Big(\ov{y}_{\ov{x}^{+}_{\e}}+\e+\CO(\e^{1+\frac{1}{2k}})+\CO(\e^{1+\la^*})+\CO(\e^{2k\la^*})\Big)+\CO(\e)\\
		& = &\displaystyle r_{\ov{x}^{+}_{\e},\e}\e+\CO(\e^{1+\frac{1}{2k}})+\CO(\e^{1+\la^*})+\CO(\e^{2k\la^*})+\CO(\e).\\
		\end{array}
		\end{equation*}
Hence, 
$$ D_{\e}(L_\e(-\e^\la))+\e^\la=\e^\la+r_{\ov{x}^{+}_{\e},\e}\e+\CO(\e^{1+\frac{1}{2k}})+\CO(\e^{1+\la^*})+\CO(\e^{2k\la^*})+\CO(\e),$$ i.e. 
\[\begin{array}{ll}\dfrac{D_{\e}(L_\e(-\e^\la))+\e^\la}{\e^\la}=&1+r_{\ov{x}^{+}_{\e},\e}\e^{1-\la}+\CO(\e^{1+\frac{1}{2k}-\la})+\CO(\e^{1+\la^*-\la})+\CO(\e^{2k\la^*-\la})\vspace{0.2cm}\\
&+\CO(\e^{1-\la}).
\end{array}\]
Using Lemma \ref{lims}, we get
\begin{equation*}
		\lim\limits_{\e\longrightarrow 0}\dfrac{D_{\e}(L_\e(-\e^\la))+\e^\la}{\e^\la}=1>0.\end{equation*}
Since $\e>0,$ there exists $\e_0>0$ such that $0<\e<\e_0\hspace{0.1cm}\Longrightarrow\hspace{0.1cm}D_{\e}(L_\e(-\e^\la))+\e^\la>0.$ Hence, $D_\e\circ L_\e(-\e^\la)>-\e^\la,$ for $\e>0$ sufficiently small. 

\begin{figure}[h]
	\begin{center}
    \begin{overpic}[scale=0.7]{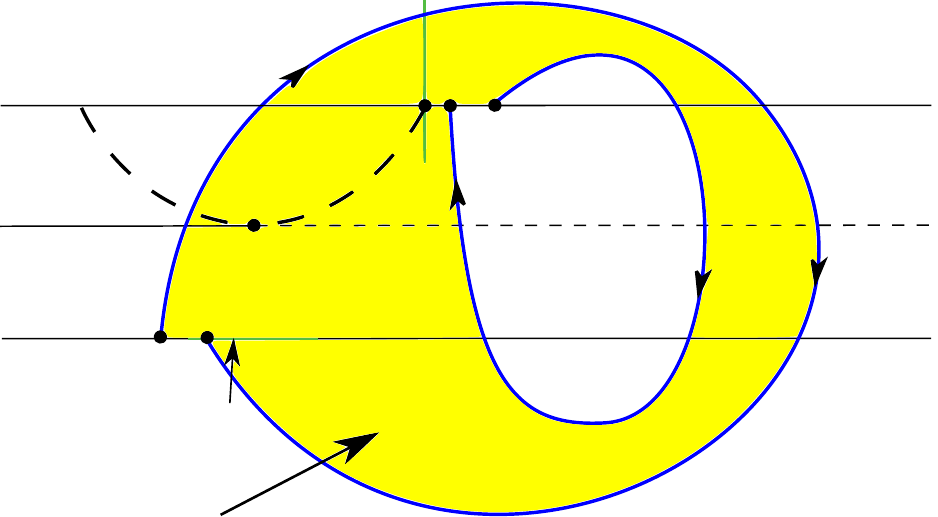}
		\put(19,21){{\tiny $D_\e\circ L_\e(-\e^\la)$ \par}}
		\put(51.5,39){$\delta$}
		\put(46.5,46){{\tiny $\widetilde{D}_\e(\delta)$ \par}}
		\put(13,15){$-\e^\la$}
		\put(23,8.5){$\tau^{s}_{\ptt}$}
		\put(101,20){$y=-\e$}
		\put(101,31){$\Sigma$}
		\put(101,45){$y=\e$}
		\put(44,57){$\tau_t^u$}
		\put(20,-3){$\mathcal{R}_\e$}
		\end{overpic}
		\caption{\tiny{The region $\mathcal{R}_\e$.}}
	\label{regionRb}
	\end{center}
	\end{figure}
Now, let $\mathcal{R}_\e$ be the region delimited by the curves $y=\e,$ $y=-\e$ and the arc-orbits connecting $(\delta,\e)$ with $(\widetilde{D}_\e(\delta),\e)$ and $(-\e^\la,-\e)$ with $(D_\e\circ L_\e(-\e^\la),-\e)$, respectively (see Figure \ref{regionRb}). It is easy to see that $\mathcal{R}_\e$ is positively invariant compact set of $V$, and has no singular points for $\e>0$ small enough. For $\e>0$  choose $q_\e\in\mathcal{R}_\e$ from the \textit{Poincar\'{e}--Bendixson Theorem} $\Gamma_\e:=\omega(q_\e)\subset \mathcal{R}_\e$ is a limit cycle of $Z_{\e}^{\Phi}$. Furthermore, $\Gamma_{\e}\subset V\subset \U$ for $\e>0$ small enough, hence $\Gamma_{\e}$ converges to $\Gamma$.

 \section{$\Sigma$-Polycycles having several visible regular-tangential singularities.}\label{sec:Polycyclesk}
In this section, we perform a qualitative analysis of $\Sigma-$polycyles having several visible regular-tangential singularities and we obtain similar results to those in Sections \ref{sec:Polycycles} and \ref{sec:polycyclesb}. Consider the nonsmooth vector field $Z=(X^+,X^-)$ which admit a $\Sigma-$polycyle having $l$ regular-tangential singularities, $p_i\in\Sigma$, of order $n_i=2k_i,$ for $i=1,\cdots, l$ satisfying one and only one of the following properties:
\begin{enumerate}
	\item[(i)] for each $i=1,\cdots, l,$ there exists a curve $\gamma_i$ connecting $p_i$ and $p_{i+1},$ oriented from $p_i$ to $p_{i+1},$ such that $\gamma_i\backslash\{p_i, p_{i+1}\}$ is a regular orbit of $Z$, $\gamma_i$ is tangent to $\Sigma$ at $p_i$ and $p_{i+1},$ where $p_{l+1}=p_1$ (see Figure \ref{figpolk1}). Locally around each $p_i$, $Z$ satisfies the same conditions of polycycles of type $(a).$
	
	\begin{figure}[h]
	\begin{center}
    \begin{overpic}[scale=0.3]{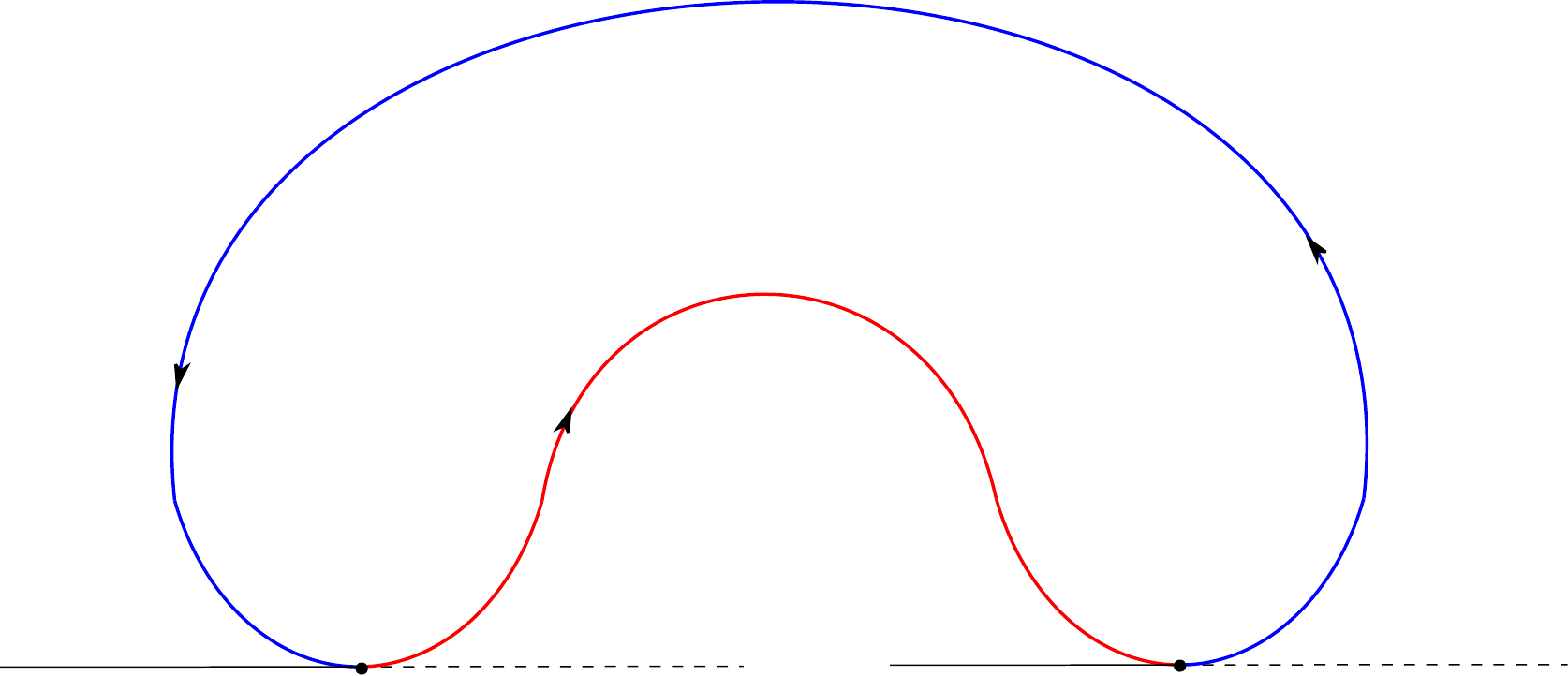}
		\put(80,36){$\Gamma$}
		\put(73,-3){$p_{2}$}
		\put(101,0){$\Sigma$}
		\put(47,18){$\gamma_1$}
		\put(47,39){$\gamma_2$}
		\put(21,-3){$p_{1}$}
		\end{overpic}
		\caption{\tiny{Example of a polyclycle of type $(i)$.}}
	\label{figpolk1}
	\end{center}
	\end{figure}
	
	\item[(ii)] for each $i=1,\cdots, l,$ there exists a curve $\gamma_i$ connecting $p_i$ and $p_{i+1},$ oriented from $p_i$ to $p_{i+1},$ such that $\gamma_i\backslash\{p_i, p_{i+1}\}$ is a regular orbit of $Z$, $\gamma_i$ is tangent to $\Sigma$ at $p_i$ and transversal to $\Sigma$ at $p_{i+1},$ where $p_{l+1}=p_1$ (see Figure \ref{figpolk2}). Locally around each $p_i$, $Z$ satisfies the same conditions of polycycles of type $(b).$
	
	\begin{figure}[h]
	\begin{center}
    \begin{overpic}[scale=0.3]{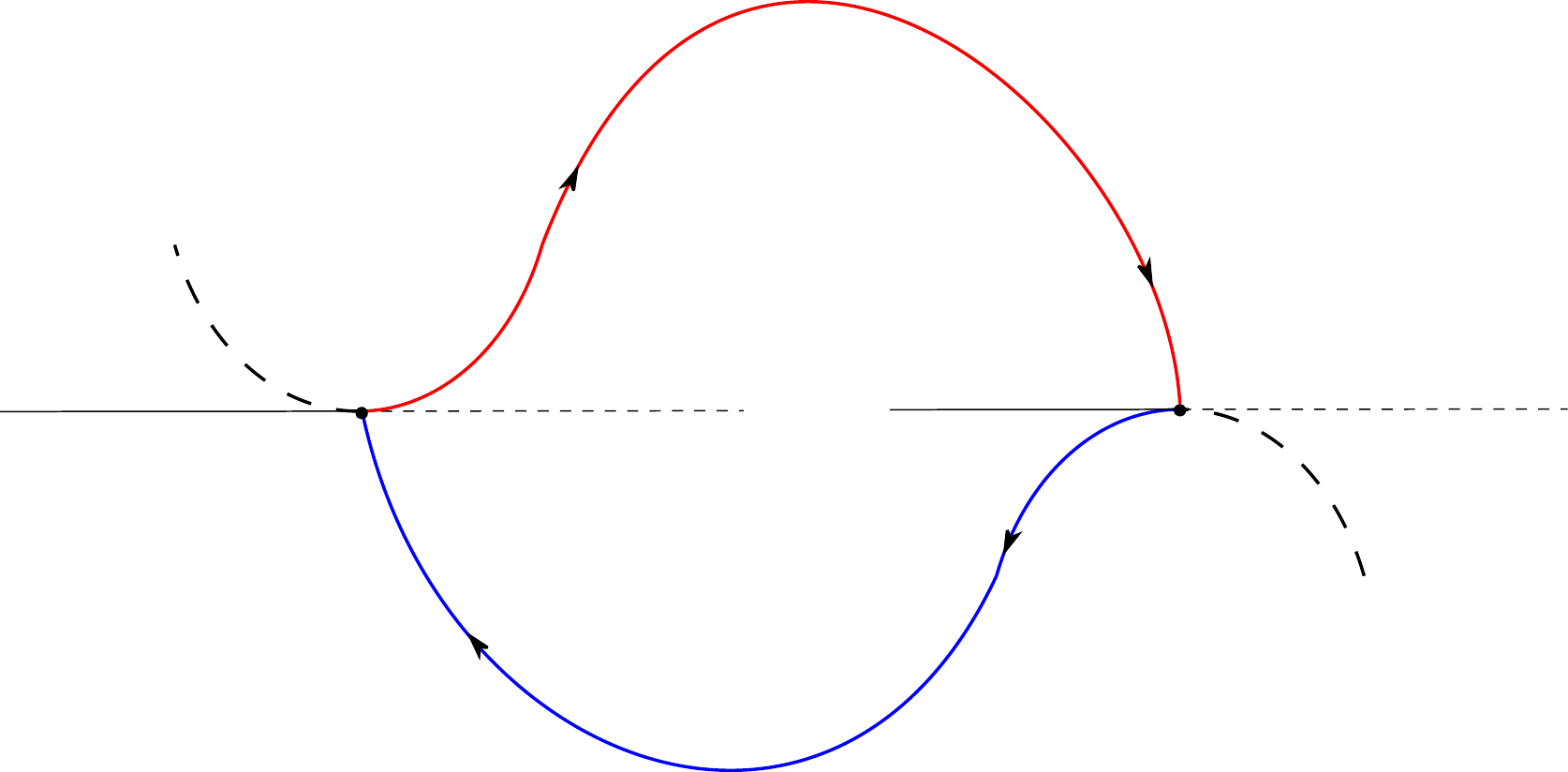}
		\put(78,34){$\Gamma$}
		\put(73,17){$p_{2}$}
		\put(101,0){$\Sigma$}
		\put(50,44){$\gamma_1$}
		\put(44,2){$\gamma_2$}
		\put(19,17){$p_{1}$}
		\end{overpic}
		\caption{\tiny{Example of a polyclycle of type $(ii)$.}}
	\label{figpolk2}
	\end{center}
	\end{figure}
	
\end{enumerate}

In what follows, without loss of generality, we assume that $h(x, y)=y,$ $p_1=(0,0),$ and $p_i=(a_i,0),$ $i=2,\cdots, l.$ 
First, we state the following lemma which is the generalization of the  lemma \ref{y0} and whose proof follows the same steps as the proof of Lemma \ref{y0}. 

\begin{lemma}\label{y0g} For each $i\in\{1,\cdots,l\}$ assume that $X^\pm$ has a visible  $n_i$-order contact with $\Sigma$ at $p_i=(a_i,0)$.
\begin{enumerate}
\item  For $\rho,\T>0,$ and $\e>0$ sufficiently small, the trajectory of $X^+$ starting at $p_i$ intersects transversally the sections $\{x=-\rho+a_i\},$ $\{x=\T+a_i\},$ and $\{y=\e\},$ respectively, at $(-\rho+a_i,\ov y_{-\rho}),$ $(\T+a_i,\ov{y}_{\T}),$ and $(\ov{x}^{\pm}_{\e}+a_i,\e),$ where $\ov y_{x}$ and $\ov{x}^{\pm}_{\e}$ were defined in \eqref{secpoints}.
\item  For $\rho,\T>0,$ and $\e>0$ small enough, the trajectory of $X^-$ starting at $p_i$ intersects transversally the sections $\{x=\rho+a_i\},$ $\{x=-\T+a_i\},$ and $\{y=-\e\},$ respectively, at $(\rho+a_i,\ov y_{-\rho}),$ $(-\T+a_i,\ov{y}_{\T}),$ and $(\widebar{x}^{\pm}_{\e}+a_i,-\e),$ where $\ov y_{x}=\frac{\al\, x^{n_i}}{n_i}+\CO(x^{n_i+1})$ and $\widebar{x}^{\pm}_{\e}=\mp\e^{\frac{1}{n_i}}\left(-\frac{n_i}{\alpha}\right)^{\frac{1}{n_i}}+\mathcal{O}(\e^{1+\frac{1}{n_i}})$, with $\alpha<0$.
\end{enumerate}
\end{lemma}

\subsection{$\Sigma-$polycycles of type $(i)$}
We start by proving that the first return map is given by $\pi_\Gamma(y)=Ky+\CO(y^2),$ for $y\in[0,\eta_1),$ with $\eta_1,K$ positive constants. For that, take $\rho,\T>0$ small enough in order that the points $q_i^{u}=(\T+a_i,\ov{y}_\T)\in W^{u}_t(p_i)$ and $q_i^{s}=(-\rho+a_i,\ov{y}_{-\rho})\in W^{s}_t(p_i)$ are contained in $U_i$, for all $i\in\{1,\cdots,l\}$. Given $\delta^{u,s}_i$ and $\eta_i$ positive numbers, since $\pi_1\circ X^+(q_i^{u,s})\neq 0$ and $\pi_1\circ X^+(p_i)\neq 0$, then 
\[\begin{array}{l}
\tau^{u}_{t,i}=\{(\T+a_i,y):y\in(\ov{y}_\T-\delta^{u}_i,\ov{y}_\T+\delta^{u}_i)\},\\
\tau^{s}_{t,i}=\{(-\rho+a_i,y):y\in(\ov{y}_{-\rho}-\delta^{s}_i,\ov{y}_{-\rho}+\delta^{s}_i)\},\\
\sigma_{p_i}=\{a_i\}\times[0,\eta_i),
\end{array}\] are transversal sections of $X^+$ at the points $q_i^{u,s}$ and $p_i$ respectively. In addition, we know by \textit{Tubular Flow Theorem} that there exist the $C^{n_i}-$diffeomorphisms $T^{u,s}_i:\sigma_{p_i}\longrightarrow\tau^{u,s}_{t,i}$ and $D_i:\tau^{u}_{t,i}\longrightarrow\tau^{s}_{t,i+1}$ such that $T_i^{u,s}(p_i)=q_i^{u,s}$ and $D_i(q_i^u)=q_{i+1}^s$, with $q_{l+1}^s=q_{1}^s$. So, expanding in Taylor series $T_i^{u,s}$ and $D_i^{-1}$ around $y=a_i$ and $y=\ov{y}_{-\rho}$, respectively, we have 
\begin{equation}\label{Tsudi}
\begin{split}
T^{s}_i(y)&=\ov{y}_{-\rho}+\kappa_{\rho,i}^{s}(y-a_i)+\CO((y-a_i)^{2}),\\
T^{u}_i(y)&=\ov{y}_\T+\kappa_{\T,i}^{u}(y-a_i)+\CO((y-a_i)^{2}),\\
D^{-1}_i(y)&=\ov{y}_\T+\frac{1}{r_{\T,\rho}^i}(y-\ov{y}_{-\rho})+\CO((y-\ov{y}_{-\rho})^2),
\end{split}
\end{equation}
%
with $\kappa_{\rho,i}^{s}=\frac{d T_i^{s}}{dy}(0),$ $\kappa_{\T,i}^{u}=\frac{d T_i^{u}}{dy}(0)$ and $r_{\T,\rho}^i=\frac{d D_i}{dy}(\ov{y}_{\T}).$ Even more, one can easy see that for each $i\in\{1,\dots,l\},$ we have that $r_{\T,\rho}^i, \kappa_{\rho,i}^{s},\kappa_{\T,i}^{u}>0$  for all $\T,\rho>0.$

\begin{proposition}
If $Z=(X^+,X^-)$ has a $\Sigma-$polycycle of type $(i)$. Then 
\begin{equation}\label{pi_K}
\frac{\p\pi_\Gamma}{\p y}(0)>0.
\end{equation}
\end{proposition}
\begin{proof}
Let $\pi_\Gamma$ be the first return map associated with the $\Sigma$-polycycle $\Gamma$ of $Z$ defined as
$$\pi_\Gamma(y)=(D_l^{-1}\circ T_1^s)^{-1}\circ T_l^u\circ(D_{l-1}^{-1}\circ T_l^s)^{-1}\circ T_{l-1}^u\circ\cdots\circ(D_1^{-1}\circ T_2^s)^{-1}\circ T_1^u(y).$$
From \eqref{Tsudi}, we get $D_{i-1}^{-1}\circ T_i^s(y)=\ov{y}_\T+\frac{\kappa_{\rho,i}^s}{r^{i-1}_{\T,\rho}}(y-a_i)+\CO((y-a_i)^{2}).$ Thus, using the \textit{Implicit Function Theorem}, we have that
	\begin{equation}\label{DT_inv_i}
	(D_{i-1}^{-1}\circ T_i^s)^{-1}(y)=a_i+\frac{r^{i-1}_{\T,\rho}}{\kappa_{\rho,i}^s}(y-\ov{y}_\T)+\CO((y-\ov{y}_\T)^2).
	\end{equation}
	Therefore, using \eqref{Tsudi} and \eqref{DT_inv_i} we conclude that 
	\begin{equation}\label{pi_nK_i}
	\pi_\Gamma(y)=\prod_{i=2}^{l+1}\frac{r^{i-1}_{\T,\rho}\kappa_{\T,i-1}^u}{\kappa_{\rho,i}^s}y+\CO{(y^{2})}.
	\end{equation}
	Consequently, taking $K:=\prod\limits_{i=2}^{l+1} K_{i-1},$ where $K_{i-1}=\frac{r^{i-1}_{\T,\rho}\kappa_{\T,i-1}^u}{\kappa_{\rho,i}^s}$ we get \eqref{pi_K}.
\end{proof}
The following lemma is a direct consequence of the lemma  \ref{rTR}.
\begin{lemma}\label{rTR_i}
Consider $r^i_{\T,\rho}$ given as in \eqref{Tsudi}, with $i\in\{1,\dots,l\}$. Then, $\lim\limits_{\T,\rho \to 0} r^i_{\T,\rho}=K_i.$
\end{lemma}
Now, for each $i\in\{1,\dots,l\}$ the exterior map associated to $Z_{\e}^{\Phi}$ around $p_ i$ is given by 
\begin{equation}\label{T_eps}
D^i_\e(y)=D_i(y)+\e S_i+\CO(\e^2),
\end{equation}
where $S_i:=\frac{\partial D^i_\e(0)}{\partial\e}.$

In what follows, we will state the main theorem of this session.
\begin{mtheorem}\label{tc2_i}
Consider a Filippov system $Z=(X^+,X^-)_{\Sigma}$ and assume that $Z$ has a polyclycle $\Gamma,$ which is of type $(i)$. For $n\geq\max\limits_{1\leq i\leq l}\{n_i\}-1,$ let $\Phi\in C^{n-1}_{ST},K,S_l$ be given as $\eqref{Phi},$ $\eqref{pi_nK_i}$ and $\eqref{T_eps}$ respectively, $\lambda_i^*=\frac{n}{1+n_i(n-1)},$ and consider the regularized system $Z_{\e}^{\Phi}$ \eqref{regula}. If $K_l+S_l-1\neq 0,$ then the following statements hold.
\begin{enumerate}
	\item[(a)] Given $0<\la<\min\left\{\frac{n_l\lambda_l^*}{n_1},\la_l^*\right\},$ if $K_l+S_l-1>0$, then there exists $\rho>0$ such that  the regularized system $Z_{\e}^{\Phi}$ does not admit limit cycles passing through the section  $\widehat H_{\rho,\la}^{\e,1}=[-\rho,-\e^{\la}]\times\{\e\},$ for $\e>0$ sufficiently small.
	\item[(b)] Given $\frac{1}{n_1}<\la<\la_l^*,$ if $K_l+S_l-1<0$, then there exists $\rho>0$ such that the regularized system $Z_{\e}^{\Phi}$ admits a unique limit cycle $\Gamma_{\e}$ passing through the section $\widehat H_{\rho,\la}^{\e,1}=[-\rho,-\e^{\la}]\times\{\e\},$ for $\e>0$ sufficiently small. Moreover, $\Gamma_{\e}$ is asymptotically stable and $\e$-close to $\Gamma.$ 
\end{enumerate}
\end{mtheorem}
\begin{proof}
For each $i\in\{1,\dots,l\}$, from Theorem \ref{ta1} of Appendix B for $\T=x_\e+a_i$, there exist $\rho^i_0>0$ and constants $\beta_i<0$ and $c,r,q>0$ such that for every $\rho\in(\e^\la,\rho^i_0],$ $\la\in(0,\la_i^*),$ and $\e>0$ sufficiently small, the flow of $Z_{\e}^{\Phi}$ defines a map $U^i_{\e}$ between the transversal sections $\widehat V_{\rho,\la}^{\e,i}=\{-\rho+a_i\}\times [\e,y_{\rho,\la}^{\e,i}]\,\,\text{ and }\,\, \widetilde V_{x_\e}^{\e,i}=\{x_\e+a_i\}\times[y_{x_\e}^{\e,i},y_{x_\e}^{\e,i}+r e^{-\frac{c}{\e^q}}]$ satisfying
\begin{equation}\label{Ue_i}
\begin{array}{cccl}
U^i_{\e}:& \widehat V_{\rho,\la}^{\e,i}& \longrightarrow& \widetilde V_{x_\e}^{\e,i}\\
&y&\longmapsto&y_{x_\e}^{\e,i}+\CO(e^{-\frac{c}{\e^q}}),
\end{array}
\end{equation}
where $y_{x_\e}^{\e,i} =\ov y_{x_\e}+\e+\mathcal{O}(\e^{n_i\la_i^*}).$ 
		
		Now, notice that there exists $\e_0>0$ such that $\widetilde V_{x_\e}^{\e,i}\subset \tau_{t,i}^u,$ for all $\e\in[0,\e_0].$ 
		In this way, for $\e\in[0,\e_0],$ we define the function 
		$$\pi_\e(y)=D^l_\e\circ U^l_{\e}\circ D^{l-1}_\e\circ U^{l-1}_{\e}\circ\dots\circ D^1_\e\circ U^1_{\e}(y).$$
		 Hence, from \eqref{T_eps} and \eqref{Ue_i}, we get
		\begin{equation}\label{Piea_i}
		\begin{array}{lllll} \pi_{\e}(y) 
		& = &\displaystyle D^l\Big(\overline{y}_{x_\e}+\e+\mathcal{O}\Big(\e^{n_l\la_l^*}\Big)\Big)+\e S_l+\CO(\e^2)\\
		& = &\displaystyle \overline{y}_{-\rho}+r^l_{x_\e,\rho}\Big(\e+\mathcal{O}\Big(\e^{n_l\la_l^*}\Big)\Big)+\mathcal{O}\Big(\e+\mathcal{O}\Big(\e^{n_l\la_l^*}\Big)\Big)^2+\e S_l+\CO(\e^2)\\
		& = &\displaystyle \overline{y}_{-\rho}+\Big(r^l_{x_\e,\rho}+S_l\Big)\e+\mathcal{O}\Big(\e^{n_l\la_l^*}\Big).\\
		\end{array}
		\end{equation}
Using \eqref{Piea_i} and \eqref{ye}, we have
		 $$\pi_\e(y)-y^{\e,1}_{\rho,\la}=\left(r^l_{x_\e,\rho}+S_l-1\right)\e+\mathcal{O}(\e\rho)-\beta_1 \e^{n_1\la}+\mathcal{O}(\e^{(n_1+1)\la})+\mathcal{O}(\e^{1+\la})+\mathcal{O}\Big(\e^{n_l\la_l^*}\Big).$$ Hence, we must study the limit $\lim\limits_{\rho,\e\rightarrow 0}\frac{\pi_\e(y)-y^\e_{\rho,\la}}{\e}$ in three distinct cases.

	 First, assume that $\la>\frac{1}{n_1}.$ Then,
		$$\frac{\pi_\e(y)-y^{\e,1}_{\rho,\la}}{\e}=r^l_{x_\e,\rho}+S_l-1+\mathcal{O}(\rho)-\beta_1 \e^{n_1\la-1}+\mathcal{O}(\e^{(n_1+1)\la-1})+\mathcal{O}(\e^{\la})+\mathcal{O}\Big(\e^{n_l\la_l^*-1}\Big).$$ Thus, by Lemma \ref{rTR_i},
		\begin{equation}\label{eql1_i}	
		\lim\limits_{\rho,\e\rightarrow 0}\frac{\pi_\e(y)-y^{\e,1}_{\rho,\la}}{\e}=K_l+S_l-1.\end{equation}
		
		Now, suppose that $\la<\frac{1}{n_1}.$ Then,
		$$\frac{\pi_\e(y)-y^{\e,1}_{\rho,\la}}{\e^{n_1\la}}=(r^l_{x_\e,\rho}+S_l-1)\e^{1-n_1\la}+\mathcal{O}(\e^{1-n_1\la}\rho)-\beta_1+\mathcal{O}(\e^{\la})+\mathcal{O}\Big(\e^{n_l\la_l^*-n_1\la}\Big).$$ Hence, by Lemma \ref{rTR_i},
		\begin{equation}\label{eql2_i}	
		\lim\limits_{\rho,\e\rightarrow 0}\frac{\pi_\e(y)-y^{\e,1}_{\rho,\la}}{\e^{n_1\la}}=-\beta_1>0.\end{equation}
		
		Finally,  assume that $\la=\frac{1}{n_1}.$ Then, 
		$$\frac{\pi_\e(y)-y^{\e,1}_{\rho,\la}}{\e}=r^l_{x_\e,\rho}+S_l-1-\beta_1+\mathcal{O}(\rho)+\mathcal{O}(\e^{\la})+\mathcal{O}\Big(\e^{n_l\la_l^*-n_1\lambda}\Big).$$ Thus, by Lemma \ref{rTR_i},
		\begin{equation}\label{eql3_i}	
		\lim\limits_{\rho,\e\rightarrow 0}\frac{\pi_\e(y)-y^{\e,1}_{\rho,\la}}{\e}=K_l+S_l-1-\beta_1.\end{equation}
	
		Now, we prove statement $(a)$ of Theorem \ref{tc2_i}. As $K_l+S_l-1>0,$ then all the above limits \eqref{eql1_i}, \eqref{eql2_i} and \eqref{eql3_i}, are strictly positive and, since $\e>0,$ there exists $\delta_0>0$ such that 
		$$0<\rho,\e<\delta_0\hspace{0.1cm}\Rightarrow\hspace{0.1cm}\pi_\e(y)-y^{\e,1}_{\rho,\la}>0.$$		
		Therefore, $\pi_\e([\e,y^{\e,1}_{\rho,\la}])\cap[\e,y^{\e,1}_{\rho,\la}]=\emptyset,$ for all $\e\in(0,\delta_0).$ This means that $\pi_\e$ has no fixed points in $[\e,y^{\e,1}_{\rho,\la}]$, i.e. the regularized system $Z_{\e}^{\Phi}$ does not admit limit cycles passing through the section $\widehat H_{\rho,\la}^{\e,1}.$
		
		 Now, we prove statement $(b)$ of Theorem \ref{tc2_i}. In this case, $\la>\frac{1}{n_1}.$  As $K_l+S_l-1<0,$ then the limit \eqref{eql1_i} is strictly negative and, since $\e>0,$  there exists $\delta_0>0$ such that 
		$0<\rho,\e<\delta_0\hspace{0.1cm}\Rightarrow\hspace{0.1cm}\pi_\e(y)-y^{\e,1}_{\rho,\la}<0.$ Consequently, $\pi_\e(y)<y^{\e,1}_{\rho,\la}.$ Moreover, from \eqref{Piea_i}, we get
		\begin{equation*}\label{eql4}	
		\lim\limits_{\e\rightarrow 0}\pi_\e(y)-\e=\ov{y}_{-\rho}>0,\end{equation*}
		for all $\rho>0.$ Since $\e>0,$  there exists $\delta_1>0$ such that 
		$0<\e<\delta_1\hspace{0.1cm}\Rightarrow\hspace{0.1cm}\pi_\e(y)-\e>0.$ Accordingly, $\pi_{\e}(y)>\e,$ for $\e>0$ sufficiently small. This means that $\pi_\e([\e,y^{\e,1}_{\rho,\la}])\subset[\e,y^{\e,1}_{\rho,\la}].$  By the {\it Brouwer Fixed Point Theorem}, we can conclude that $\pi_\e$ admits fixed points in $[\e,y^{\e,1}_{\rho,\la}]$, i.e. the regularized system $Z_{\e}^{\Phi}$ has periodic orbits passing through the section $\widehat H_{\rho,\la}^{\e,1}.$
		
Following the same steps of the proof of Theorem \ref{tc1}, it is easy to see that $\pi_\e$ is a contraction for $\e>0$ small enough. From the \textit{Banach Fixed Point Theorem}, $\pi_\e$ admits a unique asymptotically stable fixed point for $\e>0$ small enough. Therefore, $Z_{\e}^{\Phi}$ has a unique asymptotically stable limit cycle $\Gamma_{\e}$ passing through the section $\widehat H_{\rho,\la}^{\e,1},$ for $\e>0$ sufficiently small. Moreover, since $\pi_\e(y)-\ov{y}_{-\rho}=\mathcal{O}(\e)$ for all $y\in[\e,y_{\rho,\la}^{\e}]$ and $x_{\e}-\ov{x}^{+}_{\e}=\mathcal{O}(\e^{\frac{1}{2k}}),$ we get from differentiable dependency results on parameters and initial condition that $\Gamma_\e$ is $\e$-close to $\Gamma.$

\end{proof}
\subsection{$\Sigma-$polycycles of type $(ii)$}
First, we shall see that there exist positive constants $\T$ and $K$ such that the first return map 
is $\pi_\Gamma(x)=Kx^{N-1}+\CO(x^{N})$, for $x\in[0,\T],$ where $N=n_1+\dots+n_l+1$. For that, take $\e,\T>0$ small enough in order that the points $q_i^{u,\pm}=(\pm\T+a_i,\ov{y}_\T)\in W^{u}_t(p_i)$ and $q_i^{s,\pm}=(a_i,\pm\e)\in W^{s}_{\ptt}(p_i)$ are contained in $U_i$, for all $i\in\{1,\dots,l\}$. Then there exist positive numbers $\delta_{i}^{u}$ and $\delta_{i,\e}^{s}$ such that
\begin{equation*}
\begin{array}{l}
\tau^{u,\pm}_{t,i}=\{(\pm\T+a_i,y):y\in(\ov{y}_\T-\delta_{i}^{u},\ov{y}_\T+\delta_{i}^{u})\} \,\text{and}\\
\tau^{s,\pm}_{\ptt,i}=\{(x+a_i,\pm\e):x\in(-\delta_{i,\e}^{s},\delta_{i,\e}^{s})\}\\
\end{array}\end{equation*} are transversal sections of $X^+$ and $X^{-}$. In addition, $\sigma^+_{p,i}=[a_i,\T+a_i]\times\{0\}$ and $\sigma^-_{p,i}=[a_i-\T,a_i]\times\{0\}$ are  transversal sections of $X^{-}$. Moreover, by the \textit{Tubular Flow Theorem} there exist the $C^{n_i}-$diffeomorphism $T_i^{s,\pm}:\sigma^\pm_{p,i}\longrightarrow\tau^{s,\mp}_{\ptt,i}$ and $D_i^\pm:\tau^{u,\pm}_{t,i}\longrightarrow\tau^{s,\pm}_{\ptt,i+1}$ such that $T_i^{s,\pm}(p_i)=q_i^{s,\pm}$ and $D_i^\pm(q_i^{u,\pm})=q_{i+1}^{s,\pm}$. Thus, expanding $D_i^\pm$ around $y=\ov{y}_{\T}$, we get
\begin{equation}\label{Dii_ii}
D_i^\pm(y)=a_{i+1}+r^{i,\pm}_{\T,\e}(y-\ov{y}_{\T})+\CO((y-\ov{y}_{\T})^2),
\end{equation}
where $r^{i,\pm}_{\T,\e}=\frac{d D_{i}^\pm}{dy}(\ov{y}_{\T}).$

Now, expanding $T_i^{s,\pm}$ around $x=a_i$, we have 
\begin{equation}\label{Tsd2_ii}
T_i^{s,\pm}(x)=a_i+\kappa_{\e,i}^{s,\pm}(x-a_i)+\CO((x-a_i)^{2}),
\end{equation}
with $\kappa_{\e,i}^{s,\pm}=\frac{d T_i^{s,\pm}}{dx}(a_i).$ In addition, using \eqref{Dii_ii} and the \textit{Implicit Function Theorem}, we get
\begin{equation}\label{Dinv2_ii}
(D_i^{-1})^\pm(x)=\ov{y}_\T+\frac{1}{r^{i,\pm}_{\T,\e}}(x-a_{i+1})+\CO((x-a_{i+1})^2).
\end{equation}
Since $X^+$ and $X^-$ are planar vector fields, the uniqueness of solutions implies that  $r^{i,\pm}_{\T,\e},\kappa_{\e,i}^{s,\pm}>0$  for all $\T,\e>0$ sufficiently small.

We remark that the arc-orbit connecting $p_i$ with $q_i^{u,\pm}$ is contained in $\Sigma^\pm$. Then, from \cite[Theorem A]{AndGomNov19} we know that there exists a transition map $T_i^{u,\pm}:\sigma^\pm_{p,i}\longrightarrow\tau^{u,\pm}_{t,i}$ defined as
\begin{equation}\label{Tsu2_ii}
\begin{split}
T_i^{u,\pm}(x)&=\ov{y}_\T+\kappa_{\T,i}^{u,\pm}(x-a_i)^{n_i}+\CO((x-a_i)^{n_i+1}),\\
\end{split}
\end{equation}
where $\sgn(\kappa_{\T,i}^{u,\pm})=-\sgn((X^{\pm})^{n_i}h(p_i))$, i.e. $\kappa_{\T,i}^{u,+}<0$ and $\kappa_{\T,i}^{u,-}>0$.

\begin{proposition}
If $Z=(X^+,X^-)$ has a $\Sigma-$polycycle of type $(ii)$. Then 
\[
\frac{\p^{i}\pi_\Gamma}{\p x^{i}}(0)=0, \,\text{for all}\,i\in\{1,\dots,N-2\} \quad \text{and}\quad \frac{\p^{N-1}\pi_\Gamma}{\p x^{N-1}}(0)\neq 0.
\]
where $N=n_1+\dots+n_l+1.$
\end{proposition}
\begin{proof}
Notice that the first return map associated with the $\Sigma$-polycycle $\Gamma$ of $Z$
is given by
$$\pi(y)=(D_l^{-1}\circ T_1^s)^{-1}\circ T_l^u\circ(D_{l-1}^{-1}\circ T_l^s)^{-1}\circ T_{l-1}^u\circ\cdots\circ(D_1^{-1}\circ T_2^s)^{-1}\circ T_1^u(y)$$
From \eqref{Tsd2_ii} and \eqref{Dinv2_ii}, we get $D_{i-1}^{-1}\circ T_i^s(x)=\ov{y}_\T+\frac{\kappa_{\e,i}^{s\pm}}{r^{{i-1},\pm}_{\T,\e}}(x-a_i)+\CO((x-a_i)^{2}).$ Thus, using the \textit{Implicit Function Theorem}, we have that
	\begin{equation}\label{DT_inv_ii}
	(D_{i-1}^{-1}\circ T_i^s)^{-1}(y)=a_i+\frac{r^{i-1,\pm}_{\T,\e}}{\kappa_{\e,i}^{s,\pm}}(y-\ov{y}_\T)+\CO((y-\ov{y}_\T)^2).
	\end{equation}
	Therefore, using \eqref{Tsu2_ii} and \eqref{DT_inv_ii} we conclude that 
	\begin{equation*}
	\pi_\Gamma(y)=\prod_{i=2}^{l+1}\frac{r^{i-1,\pm}_{\T,\e}\kappa_{\T,i-1}^{u,\pm}}{\kappa_{\e,i}^{s,\pm}}x^{n_{i-1}}+\CO{(x^{N})}.
	\end{equation*}
	Consequently, taking $K:=\prod\limits_{i=2}^{l+1}K_{i-1},$ where $K_{i-1}=\frac{r^{i-1,\pm}_{\T,\e}\kappa_{\T,i-1}^{u,\pm}}{\kappa_{\e,i}^{s,\pm}},$ we get the result.
\end{proof}
\begin{remark}\label{Gstable_ii}
Notice that for $x\in[0,\T]$ we have that $|\pi_\Gamma(x)|<|x|,$ for some small $\T>0$. This means that $\Gamma$ is always asymptotically stable.
\end{remark}
The following lemma is a direct consequence of the lemma  \ref{lims}.
\begin{lemma}\label{lims_ii} Consider $r^{i,\pm}_{\T,\e}$ and $\kappa_{\T,i}^{u,\pm}$ given as in \eqref{Dii_ii} and \eqref{Tsu2_ii}, respectively. Then, for each $i\in\{1,\dots,l\}$
\begin{enumerate}
\item[i)] $\lim\limits_{\T \to 0}\kappa_{\T,i}^{u,\pm}=-\frac{\al_i}{n_i}.$
\item [ii)] $\lim\limits_{\T,\e \to 0}r^{i,\pm}_{\T,\e}=-\frac{n_iK_{i}}{\al_i}.$
\end{enumerate}
\end{lemma}
Now, since $D_i$ is a diffeomorphism induced by a regular orbit, we can easily see that the regularized system  $Z_{\e}^{\Phi}$ also admits an exterior map given by 
\begin{equation}\label{Tepsb_ii}
D^{i,\pm}_\e(y)=D_{i}^\pm(y)+\CO(\e).
\end{equation}
In what follows, we shall state the main result of this section.

\begin{mtheorem}\label{tc2_ii}
Consider a Filippov system $Z=(X^+,X^-)_{\Sigma}$ and assume that $Z$ has a $\Sigma-$polycycle $\Gamma$ of type $(ii).$ For $n\geq\max\limits_{1\leq i\leq l}\{n_i\}-1,$ let $\Phi\in C^{n-1}_{ST}$ be given as $\eqref{Phi},$ $\la_i^*=\frac{n}{1+n_i(n-1)},$ and consider the regularized system $Z_{\e}^{\Phi}$ \eqref{regula}. Then $Z_{\e}^{\Phi}$  admits at least a limit cycle $\Gamma_{\e},$ for $\e>0$ small enough. Moreover, $\Gamma_{\e}$ converges to $\Gamma.$
\end{mtheorem}
\begin{proof}
Let $\U$ be an open set such that $\Gamma\subset \U.$ Since $\Gamma$ is a $\Sigma-$polycycle of $Z$, then it is easy to see that there exists an open set $V\subset \U$ such that $V$ has no singular points of $Z_{\e}^{\Phi}$. Taking $(\delta,0)\in V,$ we get that $(\delta,\e)\in V,$ for $\e>0$ small enough.

Now, define the map $\widetilde{D}_\e(x)=\pi_\Gamma(x)+\CO(\e),$ for all $x\in I_\delta^\e,$ where $I_\delta^\e$ is a neighborhood of $\delta.$ By Remark \ref{Gstable_ii} we know that $\pi_\Gamma(\delta)<\delta,$ then we can conclude that $\widetilde{D}_\e(\delta)<\delta,$ for $\e>0$ sufficiently small. 

For each $i\in\{1,\dots,l\}$, by Theorem \ref{tb1} of Appendix B for $\T=a_i+\ov{x}^{+}_{\e}$,  
there exist $\rho^i_0>0,$ and constants $c,r,q>0$ such that for every $\rho\in(\e^\la,\rho^i_0],$ 
$\la\in(0,\la^*_i),$ and $\e>0$ sufficiently small, the flow of $Z_{\e}^{\Phi}$ defines a map $L^{i,+}_{\e}$ (resp. $L^{i,-}_{\e}$) between the transversal sections $\widecheck H_{\rho,\la}^{\e,i,+}=[-\rho+a_i,-\e^{\la}+a_i]\times\{-\e\}$ (resp. $\widecheck H_{\rho,\la}^{\e,i,-}=[\e^{\la}+a_i,\rho+a_i]\times\{\e\}$) and $\widecheck V_{\ov{x}^{+}_{\e}}^{\e,i,+}=\{\ov{x}^{+}_{\e}+a_i\}\times[y_{\ov{x}^{+}_{\e}}^{\e,i}-r e^{-\frac{c}{\e^q}},y_{\ov{x}^{+}_{\e}}^{\e,i}]$ (resp. $\widecheck V_{\ov{x}^{+}_{\e}}^{\e,i,-}=\{\ov{x}^{+}_{\e}+a_i\}\times[y_{\ov{x}^{+}_{\e}}^{\e,i},y_{\ov{x}^{+}_{\e}}^{\e,i}+r e^{-\frac{c}{\e^q}}]$) satisfying
\begin{equation}\label{le_ii}
\begin{array}{cccl}
L^{i,\pm}_{\e}:& \widecheck H_{\rho,\la}^{\e,i,\pm}& \longrightarrow& \widecheck V_{\ov{x}^{+}_{\e}}^{\e,i,\pm}\\
&x&\longmapsto&y_{\ov{x}^{+}_{\e}}^{\e,i}+\CO(e^{-\frac{c}{\e^q}}),
\end{array}
\end{equation}
where  $y_{\ov{x}^{+}_{\e}}^{\e,i}=\ov{y}_{\ov{x}^{+}_{\e}}+\e+\CO(\e^{1+\frac{1}{n_i}})+\CO(\e^{1+\la_i^*})+\CO(\e^{n_i\la_i^*})$.

Since $\widecheck V_{\ov{x}^{+}_{\e}}^{\e,i,\pm}\subset\tau_{t,i}^{u,\pm}$ and $\tau_{t,i}^{s,\pm}\subset\widecheck H_{\rho,\la}^{\e,i,\mp}$, for $\e>0$ sufficiently small, then we can consider the map $\pi_\e=D^{l,-}_\e\circ L^{l,-}_\e\circ\dots\circ D^{1,+}_\e\circ L^{1,+}_\e:\widecheck H_{\rho,\la}^{\e,1,+}\longrightarrow\tau^{s,-}_{\ptt,1}.$ 
 We claim that $\pi_\e(-\e^\la)>-\e^\la.$ Indeed, from \eqref{Tepsb_ii} and \eqref{le_ii}, we get

		\begin{equation*}
		\begin{array}{lllll} \pi_\e(-\e^\la) 
		& = &\displaystyle D^{l,-}_\e\Big(y^{\e,l}_{\ov{x}^{+}_{\e}}+\mathcal{O}(e^{-c/\e^q})\Big)\\
		& = &\displaystyle D_l^{-}\Big(\ov{y}_{\ov{x}^{+}_{\e}}+\e+\CO(\e^{1+\frac{1}{n_l}})+\CO(\e^{1+\la_l^*})+\CO(\e^{n_l\la_l^*})\Big)+\CO(\e)\\
		& = &\displaystyle r^{l,-}_{\ov{x}^{+}_{\e},\e}\e+\CO(\e^{1+\frac{1}{n_l}})+\CO(\e^{1+\la_l^*})+\CO(\e^{n_l\la_l^*})+\CO(\e).\\
		\end{array}
		\end{equation*}
Hence, 
$$ \pi_\e(-\e^\la)+\e^\la=\e^\la+r^{l,-}_{\ov{x}^{+}_{\e},\e}\e+\CO(\e^{1+\frac{1}{n_l}})+\CO(\e^{1+\la_l^*})+\CO(\e^{n_l\la_l^*})+\CO(\e),$$ i.e. 
\[\begin{array}{ll}\dfrac{\pi_\e(-\e^\la)+\e^\la}{\e^\la}=&1+r^{l,-}_{\ov{x}^{+}_{\e},\e}\e^{1-\la}+\CO(\e^{1+\frac{1}{n_l}-\la})+\CO(\e^{1+\la_l^*-\la})+\CO(\e^{n_l\la_l^*-\la})\vspace{0.2cm}\\
&+\CO(\e^{1-\la}).
\end{array}\]
Using Lemma \ref{lims_ii}, we get
\begin{equation*}
		\lim\limits_{\e\longrightarrow 0}\dfrac{\pi_\e(-\e^\la)+\e^\la}{\e^\la}=1>0.\end{equation*}
Since $\e>0,$ there exists $\e_0>0$ such that $0<\e<\e_0$ implies $\pi_\e(-\e^\la)+\e^\la>0.$ Hence, $\pi_\e(-\e^\la)>-\e^\la,$ for $\e>0$ sufficiently small. 

Now, let $\mathcal{R}_\e$ be the region delimited by the curves $y=\e,$ $y=-\e$ and the arc-orbits connecting $(\delta,\e)$ with $(\widetilde{D}_\e(\delta),\e)$ and $(-\e^\la,-\e)$ with $(\pi_\e(-\e^\la),-\e)$, respectively. It is easy to see that $\mathcal{R}_\e$ is positively invariant compact set of $V$, and has no singular points for $\e>0$ small enough. For $\e>0$  choose $q_\e\in\mathcal{R}_\e$ from the \textit{Poincar\'{e}--Bendixson Theorem} $\Gamma_\e:=\omega(q_\e)\subset \mathcal{R}_\e$ is a limit cycle of $Z_{\e}^{\Phi}$. Furthermore, $\Gamma_{\e}\subset V\subset \U$ for $\e>0$ small enough, hence $\Gamma_{\e}$ converges to $\Gamma$.
\end{proof}
\section*{Appendix A: The exterior map of the regularized system}\label{sec:appendix2}
In this section, we shall study the exterior map of the regularized system $Z_\e^\Phi$ and its derivative.

Let $Z=(X,Y)$ be a Filippov system, with $X,Y:V\subset\R^2\rightarrow\R^2$ vector fields of class $\mathcal{C}^{2k}$ defined on an open set $V$ of $q\in\R^2,$ and $Z_\e^\Phi$ the regularized system associated with $Z$. 
Assume that the switching manifold $\Sigma$ of $Z$ is a $\mathcal{C}^{2k}$ embedded codimension one submanifold of $V$ and let $q\in\Sigma^c$.

Notice that there exists a local $\mathcal{C}^{2k}$ diffeomorphism $\psi_{1}:U\subset\R^2\rightarrow\R^2$ defined on an open set $U$ of $q\in\R^2$ such that $\widetilde \Sigma=\psi_1(\Sigma)=h^{-1}(0),$ with $h(x,y)=x.$ Now, applying the \textit{Tubular Flow Theorem} for $(\psi_1)_*Y$ at $q$ and considering the transversal section $\widetilde \Sigma,$ there exists a local $\mathcal{C}^{2k}$ diffeomorphism $\psi_{2}$ defined on $U$ (taken smaller if necessary) such that $\widetilde Y=(\psi_2\circ\psi_1)_*Y=(1,0)$ and $\psi_{2}(\widetilde\Sigma)=\widetilde\Sigma.$ 

Since $X_1(q)>0$ then $X_1(x,y)> 0$ for all $(x,y)\in U$ (taken smaller if necessary). Consequently, we can perform a time rescaling in $X,$ thus we get $\widehat{X}(x,y)=(1,X_2(x,y)/X_1(x,y)),$ for all $(x,y)\in U.$ Hence, without loss of generality, we can assume that there exists an open set $U\subset\R^2$ of $q=(0,q_2)$ such that the Filippov system $Z=(X,Y)_{\Sigma}$ satisfies that $X\big|_{U}=(1,X_2),$ $Y\big|_{U}=(1,0),$ and $\Sigma\cap U=\{(0,y):\, y\in(q_2-\delta_U,q_2+\delta_U)\}.$

Take $\rho>0$ and $\T>0$ sufficiently small in order that the sections $\{x=-\rho\}$ and $\{x=\T\}$ is contained in $U$ and without loss of generality assume that the points $(-\rho,0)$ and $q$ are connected by a trajectory of $X.$ Thus, we can consider the following maps induced by the flow of $Z_\e^\Phi$: 
\[\begin{array}{rcl}
P_\e:\{x=-\rho\}&\rightarrow &\{x=-\e\},\\ 
T_\e:\{x=-\e\}&\rightarrow &\{x=\e\},\\
Q_\e:\{x=\e\}&\rightarrow &\{x=\T\},
\end{array}
\]
where $Q_\e=Id$ (see Figure \ref{figextmap}). The exterior map $D_\e:\{x=-\rho\}\rightarrow\{x=\T\}$ of $Z_\e^\Phi$ is defined as $D_\e(y):=Q_\e\circ T_\e\circ P_\e(y).$
 \begin{figure}[h]
	\begin{center}
    \begin{overpic}[scale=0.3]{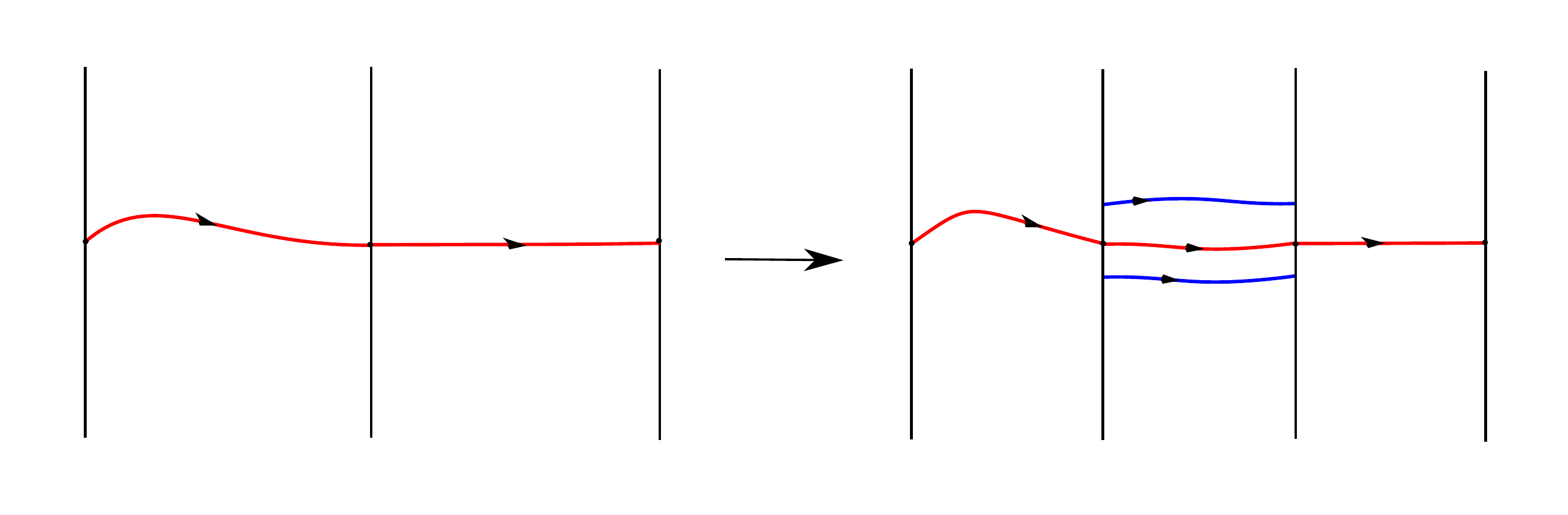}

        \put(86,21){$Q_\e$}
		\put(76,21){$T_\e$}
		\put(49,17){$\Phi$}
		\put(2,17){$0$}
		\put(63,21){$P_\e$}
		\put(29,19){$Q=Id$}
        \put(14,19){$P$}
        \put(24,19){$q_2$}
		\put(82,29){$\e$}
		\put(68,29){$-\e$}
		\put(55,29){$-\rho$}
		\put(94,29){$\T$}
		\put(2,29){$-\rho$}
		\put(42,29){$\T$}
		\put(23,29){$\Sigma$}		
		
		\end{overpic}
		\caption{\tiny{Exterior map $D_\e=Q_\e\circ T_\e\circ P_\e$ of the regularized system $Z_\e^\Phi$.}}
	\label{figextmap}
	\end{center}
	\end{figure}
	
Now, we shall compute the first derivative of $D_\e$ at $\e=0$ and $y=0$. Expanding $P_\e$ and $T_\e$  around $\e=0$, we get that
 $P_\e=P+P_1\e+\CO(\e^2)$ and $T_\e=Id+T_1\e+\CO(\e^2),$ 
where $P:\{x=-\rho\}\rightarrow\Sigma$ is the map induced by the flow of $Z,$ $P_1=\frac{\p P_\e}{\p \e}\Big|_{\e=0}$, and $T_1=\frac{\p T_\e}{\p \e}\Big|_{\e=0}$. Hence, expanding $D_\e$ around $\e=0,$ we have that
 $$D_\e(y)=P(y)+[(P_1(y)+T_1(P(y))]\e+\CO(\e^2).$$
Now, expanding $P_1+T_1\circ P$ around $y=0,$ we get
  \[\begin{array}{rcl}
  D_\e(y)&=&P(y)+[(P_1(0)+T_1(P(0))]\e+\CO(\e^2,\e y)\\
  &=& P(y)+[(P_1(0)+T_1(q_2)]\e+\CO(\e^2,\e y)\\
   &=& P(y)+S\e+\CO(\e^2,\e y),
\end{array}\]
 where $S:=P_1(0)+T_1(q_2).$ In order to determine $S,$ we shall study the maps $P_1$ and $T_1$ at $y=0$ and $y=q_2$, respectively. 
 
 First, we study the map $P_1.$ Let $\varphi_{X}$ be the flow of $X$, and  assume that $I_{(-\rho,y)}$ the maximal interval of existence of $t\mapsto\varphi_{X}(t,-\rho,y)$. Is easy to see that there exists a smooth function $t_\rho(y)$ such that $t_0(0)=0$ and $\varphi^1_{X}(t_\rho(y),-\rho,y)=0.$ Thus, we can write $P$ as follows
$$P(y)=\varphi^2_{X}(t_\rho(y),-\rho,y), \hspace{0.2cm}\text{for} \hspace{0.2cm} y\in \{x=-\rho\}.$$ 
From Implicit Function Theorem, there exists a smooth function $\tau_\rho(y,\e)$ such that $\tau_\rho(y,0)=t_\rho(y)$, $\tau_0(0,0)=0,$ and $\varphi^1_{X}(\tau_\rho(y,\e),-\rho,y)=-\e,$ for $|\e|\neq0$ sufficiently small. Hence, we can write $P_\e$ as follows
$$P_\e(y)=\varphi^2_{X}(\tau_\rho(y,\e),-\rho,y), \hspace{0.2cm}\text{for} \hspace{0.2cm}  y\in \{x=-\rho\}.$$
Notice that
$$P_1(y)=\frac{\p P_\e}{\p \e}(y)\Big|_{\e=0}=\frac{\varphi^2_{X}}{\p t}(t_\rho(y),-\rho,y)\frac{\p\tau_\rho}{\p \e}(y,0)$$
and
$$\frac{\p\tau_\rho}{\p \e}(y,0)=-\frac{1}{\frac{\varphi^1_{X}}{\p t}(t_\rho(y),-\rho,y)}.$$
Consequently, 
  \[
  P_1(0)=-\displaystyle\frac{\frac{\varphi^2_{X}}{\p t}(t_\rho(0),-\rho,0)}{\frac{\varphi^1_{X}}{\p t}(t_\rho(0),-\rho,0)}
   = -\displaystyle\frac{X_2(q)}{1}
    = -X_2(q).
\]

Second, we study the map $T_1.$ For this, consider the differential system associated to the regularized system $Z_{\e}^{\Phi}$ \eqref{regula} given by
\begin{equation}\label{reg1sys}
\left\lbrace\begin{array}{l}
\dot{x} = 1,\vspace{0.2cm}\\
\dot{y} =   \dfrac{1}{2} X_2(x,y)\left(1+\Phi(x/\e)\right),
\end{array}\right.
\end{equation}
for $(x,y)\in U$ and $\e>0$ small enough. 
Notice that system \eqref{reg1sys}, restricted to the band of regularization $-\e\leqslant x \leqslant \e,$ is as a {\it slow-fast problem}. Indeed, taking $v=x/\e,$ we get the so-called {\it slow system},
\begin{equation}\label{regsys}
\left\lbrace\begin{array}{l}
\e\dot{v} = 1,\vspace{0.2cm}\\
\dot{y} =   \dfrac{1}{2} X_2(\e v,y)\left(1+\phi(v)\right),
\end{array}\right.
\end{equation}
defined for $-1 \leqslant v \leqslant 1.$ Performing the time rescaling $t=\e\tau,$ we obtain the so-called {\it fast system},
\begin{equation}\label{regsys}
\left\lbrace\begin{array}{l}
v' = 1,\vspace{0.2cm}\\
y' = \dfrac{\e}{2} X_2(\e v,y)\left(1+\phi(v)\right).
\end{array}\right.
\end{equation}
Now, we know that the equation for the orbits of system \eqref{regsys} is given by
\begin{equation}\label{eq29}
\frac{dy}{dv}=\e F(v,y,\e),
\end{equation}
where 
$$F(v,y,\e)=\frac{ X_2(\e v,y)\left(1+\phi(v)\right)}{2}.$$ 
Expanding $F$ around $\e=0,$ we get
$$F(v,y,\e)=\frac{ X_2(0,y)\left(1+\phi(v)\right)}{2}+\CO(\e).$$ 
In addition, the solution of differential equation \eqref{eq29} with initial condition $y_\e(-1,q_2)=q_2$ is given by 
 \[\begin{array}{rcl}
y_\e(v,q_2)&=&q_2+\e\displaystyle\int_{-1}^v F(s,y_\e(s,q_2),\e)ds\\
&=&q_2+\e\displaystyle\int_{-1}^v F(s,y_0(s,q_2),0)ds+\CO(\e^2)\\
  &=& q_2+\displaystyle\frac{\e}{2}\int_{-1}^v X_2(q)\left(1+\phi(s)\right)ds+\CO(\e^2).
 \end{array}
 \]
Notice that 
 \[\begin{array}{rcl}
 T_\e(q_2)&=& y_\e(1,q_2)\\
 &=&q_2+\e\displaystyle\int_{-1}^1 F(s,y_0(s,q_2),0)ds+\CO(\e^2)\\
  &=& q_2+\displaystyle\frac{\e}{2}\int_{-1}^1 X_2(q)\left(1+\phi(s)\right)ds+\CO(\e^2).
\end{array}
 \]
Accordingly,
 \[
 \begin{array}{rcl}
 T_1(q_2)&=&\displaystyle\frac{1}{2}\int_{-1}^1 X_2(q)\left(1+\phi(s)\right)ds\\
 &=& \displaystyle\frac{X_2(q)}{2}\left(2+\int_{-1}^1 \phi(s)ds\right).
 \end{array}
 \]
Therefore, 
 \[
 \begin{array}{rcl}
 S&=& P_1(0)+T_1(q_2)\\
  &=& -X_2(q)+\displaystyle\frac{X_2(q)}{2}\left(2+\int_{-1}^1 \phi(s)ds\right)\\
 &=& \displaystyle\frac{X_2(q)}{2}\int_{-1}^1 \phi(s)ds.
 \end{array}
 \]
 
 \section*{Appendix B: Upper and Lower Transition Maps}\label{sec:appendix}

In this section, we need to know how is the behaviour of the trajectories of the regularized system $Z_{\e}^{\Phi}$ around visible regular-tangential singularities. In \cite{NovRon2019} was proven how is the dynamic of $Z_{\e}^{\Phi}$ near a visible regular-tangential singularity. Before stating these results we need to introduce some notations.

Let $Z=(X^+,X^-)_{\Sigma},$ be a Filippov system where $X^{\pm}:V\subset\R^2\rightarrow\R^2$ is $\mathcal{C}^{2k}$ vector fields with  $k\geq 1$ and $V$ an open set of $\R^2$. Consider $\Sigma$ a $\mathcal{C}^{2k}$ embedded codimension one submanifold of $V.$ Assume that $X^+$ has a visible  $2k$-multiplicity contact with $\Sigma$ at $p$ and that $X^-$ is pointing towards $\Sigma$ at $p.$ Performing a local change of coordinates and a rescaling of time, we get that $p=(0,0)$ and $h(x,y)=y$. Therefore, without loss of generality, we can assume that the Filippov system $Z=(X^+,X^-)_{\Sigma}$ satifies
\begin{itemize}
\item[{\bf (H)}] $X^+$ admits a visible  $2k$-multiplicity contact with $\Sigma$ at $p,$ $X_1^+(p)>0,$ and there exists an open set $U\subset\R^2$ of $p$ such that $X^-\big|_U=(0,1)$ and $\Sigma\cap U=\{(x,0):\, x\in(-x_U,x_U)\}.$  
\end{itemize}
The following result was proven in \cite{NovRon2019}. It establishes the intersection between the trajectory of $X^+$ (satisfying {\bf (H)}) starting at  $p$ with some sections.
\begin{lemma}\cite[Lemma 1]{NovRon2019}\label{y0}
Assume that $X^+$ satisfies hypothesis {\bf (H)}. For $\rho>0,$ $\T>0,$ and $\e>0$ sufficiently small, the trajectory of $X^+$ starting at $p$ intersects transversally the sections $\{x=-\rho\},$ $\{x=\T\},$ and $\{y=\e\},$ respectively, at $(-\rho,\ov y_{-\rho}),$ $(\T,\ov{y}_{\T}),$ and $(\ov{x}^{\pm}_{\e},\e),$ where
\begin{equation}\label{secpoints}
\ov y_{x}=\frac{\al\, x^{2k}}{2k}+\CO(x^{2k+1}) \quad \text{and}\quad \ov{x}^{\pm}_{\e}=\pm\e^{\frac{1}{2k}}\left(\frac{2k}{\alpha}\right)^{\frac{1}{2k}}+\mathcal{O}(\e^{1+\frac{1}{2k}}).
\end{equation}
for some positive constant $\alpha.$ 
\end{lemma}

 The same way, we state the following lemma, which  establishes the intersection between the trajectory of $X^-$ (satisfying {\bf (H)}) starting at $p$ with the section $\{y=-\e\}$.
\begin{lemma}
Assume that $X^-$  satisfies hypothesis {\bf (H)}. For $\e>0$ sufficiently small, the trajectory of $X^-$ starting at $p$ intersects transversally the section $\{y=-\e\},$ at $(0,-\e).$
\end{lemma}
\begin{proof}
Consider the differential equation induced by the vector field $X^-$
\begin{equation}\left\lbrace\begin{array}{rl}\label{cs-}
x'  = & 0,\\
y'  = & 1.\\    
\end{array}\right.\end{equation}
Denote by $(x(t),y(t))$ the solution of system \eqref{cs-} satisfying $x(0)=0$ and $y(0)=0.$ Thus, $x(t)=0$ and $y(t)=t.$ 

Hence, taking $\e>0$ sufficiently small, we conclude that the trajectory of $X^-$ starting at $p$ intersects the section $\{y=-\e\}$ at the point $(0,-\e).$ This intersection is transversal, because $ X_2^-(x,y)=1$ for every $(x,y)\in U.$
\end{proof}

Now, given $\Phi\in C^{n-1}_{ST}$ as \eqref{Phi}, with $k\geqslant1,$ and $n\geqslant 2k-1$, define
\begin{equation*}\label{xe}
x_{\e}=\e^{\la^*}\eta+\mathcal{O}\left(\e^{\la^*+\frac{1}{1+2k(n-1)}}\right),
\end{equation*}
where $\la^*\defeq \frac{n}{1+2k(n-1)}$ and $\eta$ is a constant satisfying
 \begin{equation*}\label{sigmakn}
\eta>\left\lbrace\begin{array}{lllll}
0  & if & n>2k-1,\\
-\left(\dfrac{\partial_y X_2^+(p)}{\alpha}\right)^{\frac{1}{2k-1}} & if & n=2k-1 \hspace{0.1cm}\text{and}\hspace{0.1cm} k\neq 1,\\      
\end{array}\right.
\end{equation*} 
for some positive constant $\alpha$ and
\begin{equation}\label{ye}
\begin{array}{l}
y^\e_{\rho,\la}=\overline{y}_{-\rho}+\e+\mathcal{O}(\e \rho)+\beta \e^{2k\la}+\mathcal{O}(\e^{(2k+1)\la})+\mathcal{O}(\e^{1+\la}), \vspace{0.1cm}\\

\displaystyle y_{\T}^{\e} =\ov y_{\T}+\e+\mathcal{O}(\e\T)+\sum_{i=1}^{2k-1}\mathcal{O}(\T^{2k+1-i} x_\e^i)+\mathcal{O}(x_{\e}^{2k}),
\end{array}
\end{equation}
where $\la\in(0,\la^*),$ $\T$ and $\rho>0$ are parameters depending (eventually) on $\e,$ $\beta$ is a parameter such that $\sgn(\beta)=-\sgn((X^+)^{2k}h(p))<0,$ and $\ov y_{-\rho}$ and $\ov y_{\T}$ are given by Lemma \ref{y0}.

We are already prepared to establish the two main theorems of this section.
\begin{theorem}\cite[Theorem 1]{NovRon2019}\label{ta1}
Consider a Filippov system $Z=(X^+,X^-)_{\Sigma}$ and assume that $X^+$ satisfies hypothesis {\bf (H)} for some $k\geqslant 1.$ For $n\geqslant 2k-1,$ let $\Phi\in C^{n-1}_{ST}$ be given as \eqref{Phi} and consider the regularized system $Z_{\e}^{\Phi}$ \eqref{regula}. Then, there exist $\rho_0,\T_0>0,$ and constants  $\beta<0$ and $c,r,q>0,$ for which the following statements hold for every $\rho\in(\e^\la,\rho_0],$ $\T\in[x_\e,\T_0],$ $\la\in(0,\la^*),$ with $\e>0$ sufficiently small.
\begin{itemize}
\item[(a)] The vertical segments 
\[
\widehat V_{\rho,\la}^{\e}=\{-\rho\}\times [\e,y_{\rho,\la}^{\e}]\,\,\text{ and }\,\, \widetilde V_{\T}^{\e}=\{\T\}\times[y_\T^\e,y_\T^\e+r e^{-\frac{c}{\e^q}}]
\]
and the horizontal segments 
\[
\widehat H_{\rho,\la}^{\e}=[-\rho,-\e^{\la}]\times\{\e\}\,\,\text{ and }\,\, \overleftarrow{H}_{\e}=[x_{\e}-r e^{-\frac{c}{\e^q}},x_\e]\times\{\e\}
\]
are transversal sections for $Z_{\e}^{\Phi}.$

\item[(b)] The flow of $Z_{\e}^{\Phi}$ defines a map $U_{\e}$ between the transversal sections $\widehat V_{\rho,\la}^{\e}$ and $\widetilde V_{\T}^{\e}$ satisfying
\[
\begin{array}{cccl}
U_{\e}:& \widehat V_{\rho,\la}^{\e}& \longrightarrow& \widetilde V_{\T}^{\e}\\
&y&\longmapsto&y_{\T}^{\e}+\CO(e^{-\frac{c}{\e^q}}).
\end{array}
\]
\end{itemize}
The map $U_\e$ is called \textit{Upper Transition Map} of the regularized system $Z_{\e}^{\Phi}.$
\end{theorem}
		
	 \begin{theorem}\cite[Theorem 2]{NovRon2019}\label{tb1}
	 Consider a Filippov system $Z=(X^+,X^-)_{\Sigma}$ and assume that $X^+$ satisfies hypothesis {\bf (H)} for some $k\geqslant 1.$ For $n\geqslant 2k-1,$ let $\Phi\in C^{n-1}_{ST}$ be given as \eqref{Phi} and consider the regularized system $Z_{\e}^{\Phi}$ \eqref{regula}. Then, there exist $\rho_0,\T_0>0,$ and constants $c,r,q>0,$ for which the following statements hold for  every $\rho\in(\e^\la,\rho_0],$ $\T\in[x_\e+re^{-\frac{c}{\e^q}},\T_0],$ $\la\in(0,\la^*),$ with $\e>0$ sufficiently small.
	 \begin{itemize}
\item[(a)] The vertical segment
\[
\widecheck V_{\T}^{\e}=\{\T\}\times[y_\T^\e-r e^{-\frac{c}{\e^q}},y_\T^\e]
\]
and the horizontal segments 
\[
\widecheck H_{\rho,\la}^{\e}=[-\rho,-\e^{\la}]\times\{-\e\}\,\,\text{ and }\,\, \overrightarrow{H}_{\e}=[x_{\e},x_{\e}+r e^{-\frac{c}{\e^q}}]\times\{\e\}
\]
are transversal sections for $Z_{\e}^{\Phi}.$

\item[(b)] The flow of $Z_{\e}^{\Phi}$ defines a map $L_{\e}$ between the transversal sections $ \widecheck H_{\rho,\la}^{\e}$ and $\widecheck V_{\T}^{\e},$ namely
\[
\begin{array}{cccl}
L_{\e}:& \widecheck H_{\rho,\la}^{\e}& \longrightarrow& \widecheck V_{\T}^{\e}\\
&x&\longmapsto&y_{\T}^{\e}+\CO(e^{-\frac{c}{\e^q}}).
\end{array}
\]
\end{itemize}
The map $L_\e$ is called \textit{Lower Transition Map} of the regularized system $Z_{\e}^{\Phi}.$
\end{theorem}
 
\section*{Acknowledgments}

DDN is partially supported by FAPESP grants 2018/16430-8, 2018/ 13481-0, and 2019/10269-3, and by CNPq grant 306649/2018-7 and 438975/ 2018-9. GR is supported by São Paulo Research Foundation
(FAPESP) grant 2020/06708-9.

\bibliographystyle{abbrv}
\bibliography{references1}

\end{document}